\documentclass{article}%
\usepackage{amsfonts}
\usepackage{amsmath}
\usepackage{amssymb}
\usepackage{graphicx}%
\RequirePackage[a4paper,top=2.54cm,bottom=2.54cm,left=1.90cm,right=1.90cm,%
                headsep=1em,includehead,includefoot]{geometry}
\RequirePackage{float}
\RequirePackage{indentfirst}
\providecommand{\U}[1]{\protect\rule{.1in}{.1in}}
\newtheorem{theorem}{Theorem}

\newtheorem{corollary}[theorem]{Corollary}

\newtheorem{definition}[theorem]{Definition}

\newtheorem{lemma}[theorem]{Lemma}

\newtheorem{proposition}[theorem]{Proposition}
\newtheorem{remark}[theorem]{Remark}

\newenvironment{proof}[1][Proof]{\noindent\textbf{#1.} }{\ \rule{0.5em}{0.5em}}

\providecommand{\keywords}[1]
{
  \small	
  \textbf{\textit{Keywords:}} #1
}

\begin{document}
\title{Inviscid limit for Stochastic Navier-Stokes Equations under general initial conditions}
\author{Eliseo Luongo$^{1}$ \\ $^1$ Scuola Normale Superiore, Piazza dei Cavalieri,7, 56126 Pisa, Italy \\ Email: eliseo.luongo@sns.it}
\maketitle 

\begin{abstract}
  We consider in a smooth and bounded two dimensional domain the convergence in the $L^2$ norm, uniformly in time, of the solution of the stochastic Navier-Stokes equations with additive noise and no-slip boundary conditions to the solution of the corresponding Euler equations. We prove, under general regularity on the initial conditions of the Euler equations, that assuming the dissipation of the energy of the solution of the Navier-Stokes equations in a Kato type boundary layer, then the inviscid limit holds.
\end{abstract}

\section{Introduction}
The study of the inviscid limit of the solutions of the Navier-Stokes equations is a classical topic in fluid mechanics. The Euler equations have very large classes of weak solutions, including non-dissipative ones \cite{bardos2014non}, but the inviscid limit can in some cases furnish a selection principle \cite{bardos2012vanishing}. In the case of domains without boundary several results are available in the deterministic case, see, for instance \cite{constantin1988navier}, \cite{da2009sharp}, \cite{golovkin1966vanishing}, \cite{kato1972nonstationary}, \cite{mcgrath1968nonstationary}. In the case of domains with boundary the difficulty of the problem changes drastically considering different boundary conditions also in the two dimensional case. Indeed, if we consider the so called Navier boundary conditions, some results are available both in the deterministic and in the stochastic case (see for example \cite{cipriano2015inviscid}, \cite{lions1996mathematical}). The no-slip boundary conditions are more challenging. This is due to the appearance of the boundary layer. So far, only few results are available in this framework. They can be splitted in two macro-categories:\begin{enumerate}
    \item Conditioned results, namely proving that if the solution of the Navier-Stokes equations has a particular behavior in the boundary layer, then the inviscid limit holds true. These are the most common kind of results available for what concern the inviscid limit with no-slip boundary conditions. See for instance \cite{constantin2015inviscid}, \cite{kato1984remarks}, \cite{temam1997behavior}, \cite{wang2001kato}.
    \item Unconditioned results. They are based on strong assumptions about the symmetry of the domain and of the data \cite{lopes2008vanishing}, \cite{lopes2008vanishing2}, or real analytic data \cite{sammartino1998zero}, or the vanishing of the Eulerian initial vorticity in a neighborhood of the boundary \cite{maekawa2014inviscid}.
\end{enumerate}
The results of this paper go in the first direction. In particular our goal is to generalize the results of \cite{kato1984remarks} to the stochastic framework and to not classical solutions of the Euler equations. We consider the stochastic Navier-Stokes equations with additive noise and no-slip boundary conditions in a smooth, bounded, two dimensional domain, proving that, under suitable assumptions on the behavior of their solutions in the boundary layer and of the additive noise, we have strong convergence to the solution of the deterministic Euler equations for vanishing viscosity.\\ In section \ref{Functional Setting} we introduce the mathematical problem, giving some well known results about the well posedness of the stochastic Navier-Stokes equations with additive noise and the Euler equations, and stating our main theorems. In sections \ref{Sec Thm strong strong } and \ref{Sec Theorem weak strong} we prove our main theorems. Lastly in section \ref{Sec Rem} we add some deterministic results related to Theorem \ref{P.L. Lions} and Theorem \ref{Theorem weak strong}.\\ Theorem \ref{Thm strong strong} and Theorem \ref{Theorem weak strong} can be seen as introductory results for the analysis of the zero noise-zero viscosity limit following the idea of \cite{bafico1982small}. These kind of results are relevant for the analysis of a selection principle for the solutions of the Euler equations.
\section{Main Results}\label{Functional Setting} Let $D\subseteq \mathbb{R}^2$ be smooth and bounded, $T>0$ fixed and $\left(\Omega,\mathcal{F},\mathcal{F}_t,\mathbb{P}\right)$ a filtered probability space.\\
Let $Z$ be a separable Hilbert space, denote by $L^2(\mathcal{F}_{t_0},Z)$ the space of square integrable random variables with values in $Z$, measurable with respect to $\mathcal{F}_{t_0}$. Moreover, denote by $C_{\mathcal{F}}\left(  \left[  0,T\right]  ;Z\right)  $ the space
of continuous adapted processes $\left(  X_{t}\right)  _{t\in\left[
0,T\right]  }$ with values in $Z$ such that
\[
\mathbb{E}\left[  \sup_{t\in\left[  0,T\right]  }\left\Vert X_{t}\right\Vert
_{Z}^{2}\right]  <\infty
\]
and by $L_{\mathcal{F}}^{2}\left(  0,T;Z\right)  $ the space of progressively
measurable processes $\left(  X_{t}\right)  _{t\in\left[  0,T\right]  }$ with
values in $Z$ such that
\[
\mathbb{E}\left[  \int_{0}^{T}\left\Vert X_{t}\right\Vert _{Z}^{2}dt\right]
<\infty.
\]

Denote by $L^{2}\left(  D;\mathbb{R}^2\right)  $ and $H^{k}\left(  D;\mathbb{R}^2 \right)  $ the usual
Lebesgue and Sobolev spaces and by $H_{0}^{k}\left(  D;\mathbb{R}^2\right)  $ the closure
in $W^{k}\left(  D\right)  $ of smooth compact support vector valued functions. Set
$$H=\{f \in L^2(D;\mathbb{R}^2),\ \operatorname{div}f=0,\ f\cdot n|_{\partial D}=0\},\ V=H_0^1(D;\mathbb{R}^2)\cap H,\ D(A)=H^2\cap V.  $$ We denote by $\langle\cdot,\cdot\rangle$ and $\lVert\cdot\rVert$ the inner product and the norm in $H$ respectively.\\ 
Denote by $P$ the projection of $L^2\left(D;\mathbb{R}^2\right)$ on $H$ and define the unbounded linear operator $A:D(A)\subseteq H\rightarrow H$ by the identity \begin{align*}
    \langle A v, w\rangle=\langle \Delta v, w \rangle
\end{align*}
for all $v \in D(A),\ w \in H$. $A$ will be called the Stokes operator. It is well known (see for example \cite{temam2001navier}) that $A$ generates an analytic semigroup of negative type on $H$ and moreover $V=D\left(\left(-A)^{1/2}\right)\right).$\\
Let us consider a sequence of real Brownian motions $\{W^k_t\}_{k=1}^N$ adapted to $\mathcal{F}_t$ and a sequence of functions $\{\sigma_k\}_{k=1}^N\subseteq D(A)$. Let us, moreover, assume that $u_0^{\nu}\in L^2(\mathcal{F}_{0},H)$.\\ Let us consider the stochastic Navier-Stokes equations below. Some physical motivations for the introduction of this model can be found in \cite{flaWas}.\begin{align}
\begin{cases}
    d u^{\nu} &=-(-\nu\Delta u^{\nu}+\nabla u^{\nu}\cdot u^{\nu}+\nabla p^{\nu})dt+ \nu^{\frac{1}{2}}\sum_{k=1}^N \sigma_k dW^k_t\ t\in [0,T] \\
    u^{\nu}(0) & =u_0^{\nu}.    
\end{cases}
\label{additive noise}
\end{align} 

\begin{definition}\label{weak solution}
Given  $u_{0}^{\nu}\in L^{2}(\mathcal{F}_{0},H)$, we say that a stochastic process $u^{\nu}$ 
is a weak solution of equation (\ref{additive noise}) if
$$
u^{\nu} \in C_{\mathcal{F}}([0,T];H)\cap L_{\mathcal{F}}^{2}(0,T;V)
$$
and for every $\phi\in D(A)$, we have
\begin{align*}
    \langle u^{\nu}(t),\phi\rangle-\int_0^t b(u^{\nu}(s),\phi,u^{\nu}(s))\ ds=\langle u^{\nu}_0, \phi\rangle+\nu\int_0^t \langle u^{\nu}(s),A\phi\rangle\ ds+\nu^{\frac{1}{2}}\sum_{k=1}^N\langle\sigma_k,\phi\rangle W^k_t,
\end{align*}
for every $t \in[0,T],\ \mathbb{P}-a.s.$
\end{definition}
Under previous assumptions on the coefficient $\sigma_k$, equation (\ref{additive noise}) is well posed. Indeed the following theorem holds, see \cite{flaWas}.
\begin{theorem}\label{Well posed additive noise }
If $u_0^{\nu}\in L^2(\mathcal{F}_0,H),\ \{\sigma_k\}_{k=1}^N\subseteq D(A)$, there exists a unique weak solution of equation (\ref{additive noise}). Moreover the following relations hold:
\begin{align}
    \mathbb{E}\left[\lVert u^{\nu}(t)\rVert^2\right]+2\nu\int_0^t\mathbb{E}\left[\lVert u^{\nu}(s)\rVert_V^2 \right]\ ds=\mathbb{E}\left[\lVert u_0^{\nu}\rVert^2\right]+t\nu \sum_{k=1}^N \lVert \sigma_k\rVert^2 \label{energy eq}
\end{align} \begin{align}
\mathbb{E}\left[sup_{t\in [0,T]}\lVert u^{\nu}(t)\rVert^2\right]& \leq \mathbb{E}\left[\lVert u_0^{\nu}\rVert_{L^2(D)}^2\right]+T\nu^{\frac{1}{2}} \sum_{k=1}^N \lVert \sigma_k\rVert^2+ K\nu^{\frac{1}{2}} \sum_{k=1}^N\mathbb{E}\left[\int_0^T\langle u^{\nu}(s),\sigma_k\rangle^2 \ ds\right]^{\frac{1}{2}}, \label{uniform energy}   
\end{align}
\begin{align}
    \lVert u^{\nu}(t)\rVert^2+2\nu\int_0^t\lVert\nabla u^{\nu}(s)\rVert_{L^2(D)}^2\ ds= \lVert u^{\nu}_0\rVert^2+t\nu\sum_{k=1}^N\lVert \sigma_k\rVert^2
    +2\nu^{\frac{1}{2}}\sum_{k=1}^N\int_0^t\langle u^{\nu}(s),\sigma_k\rangle dW^k_s.\label{Ito}
\end{align}
where $K$ is independent from $\nu$. 
\end{theorem}
Equation (\ref{energy eq}) will be called energy equality in the following, instead equation (\ref{Ito}) will be called Itô formula.\\
For our purposes we will need a different relation satisfied by $u^{\nu}$ that will be clarified by the following lemma.

\begin{lemma}\label{Weak solution impl}
Under the same assumptions of Theorem \ref{Well posed additive noise }, if $u^{\nu}$ 
is a weak solution of equation (\ref{additive noise}), then for each $\phi \in C([0,T]; V)\cap C^1([0,T];H)$ \begin{align*}& \langle u^{\nu}(t),\phi(t)\rangle=\langle u^{\nu}(0),\phi (0)\rangle+\int_0^t\langle u^{\nu}(s),\partial_s\phi(s)\rangle\ ds\\& -\nu\int_0^t \langle(-A)^{\frac{1}{2}}u^{\nu}(s), (-A)^{\frac{1}{2}}\phi(s)\rangle\ ds  +\int_0^t b(u^{\nu}(s),\phi(s),u^{\nu}(s))\ ds\\& + \nu^{\frac{1}{2}}\sum_{k}\langle\sigma_k, \phi(t)\rangle W^k_t-\nu^{\frac{1}{2}}\sum_{k=1}^N\int_0^t \langle \sigma_k, \phi(s)\rangle W^k_s \ ds\
\end{align*}
for every $t \in[0,T],\ \mathbb{P}-a.s.$
\end{lemma}
\begin{proof}
Thanks to the regularity of the weak solution $u^{\nu}$, by density we have that for each $\phi \in V$ \begin{align*}
    \langle u^{\nu}(t),\phi\rangle-\int_0^t b(u^{\nu}(s),\phi,u^{\nu}(s))\ ds &=\langle u^{\nu}(0), \phi\rangle+\nu^{\frac{1}{2}}\sum_{k=1}^N\langle \sigma_k,\phi\rangle W^k_t\\ & -\nu\int_0^t \langle(-A)^{\frac{1}{2}}u^{\nu}(s),(-A)^{\frac{1}{2}}\phi\rangle\ ds,
\end{align*}
for every $t \in[0,T],\ \mathbb{P}-a.s.$ Let now $\phi(t)\in C^{1}([0,T];H)\cap
C([0,T];V)$. Let, moreover, $\pi=\{0=t_{0}<t_{1}<\dots<T_{n}=T\}$ be a partition of 
$[0,T]$. Thus, using the identities
\[
\langle u^{\nu}(t_{i+1}),\phi(t_{i+1})\rangle-\langle u^{\nu}(t_{i+1}),\phi(t_{i})\rangle=\int_{t_{i}%
}^{t_{i+1}} \langle u^{\nu}(t_{i+1}),\partial_{s} \phi(s)\rangle \ ds,
\]
\[
\langle \sigma_k W^k_{t_{i+1}},\phi(t_{i+1})\rangle-\langle \sigma_k W^k_{t_{i+1}},\phi(t_{i})\rangle=\int_{t_{i}%
}^{t_{i+1}} \langle \sigma_k W^k_{t_{i+1}},\partial_{s} \phi(s)\rangle \ ds,
\]
we get
\begin{align*}
\langle u^{\nu}(t_{i+1}),\phi(t_{i+1})\rangle& =\langle u^{\nu}(t_{i}),\phi(t_{i})\rangle-\int_{t_{i}%
}^{t_{i+1}}\langle (-A)^{\frac{1}{2}}u^{\nu}(s),(-A)^{\frac{1}{2}}\phi(t_{i})\rangle\ ds\\
& + \int_{t_i}^{t_{i+1}}b(u^{\nu}(s),\phi(t_i),u^{\nu}(s))\ ds\\ &+\int_{t_{i}}^{t_{i+1}}\langle u^{\nu}(t_{i+1}),\partial_{s}\phi(s)\rangle\ ds\\ &-\nu^{\frac{1}{2}}\sum_{k=1}^N%
\int_{t_{i}}^{t_{i+1}}\langle \sigma_k W^{k}(t_{i+1}),\partial_s\phi(s)\rangle\ ds\\ &+\nu^{\frac{1}{2}} \sum_{k=1}^N\left(\langle\sigma_k W^k_{t_{i+1}},\phi(t_{i+1}) \rangle-\langle \sigma_kW^k_{t_{i}},\phi(t_{i}) \rangle\right).
\end{align*}
It implies 
\begin{align*}
\langle u^{\nu}(T),\phi(T)\rangle& =\langle u^{\nu}(0),\phi(0)\rangle-\int_{0%
}^{T}\langle (-A)^{\frac{1}{2}}u^{\nu}(s),(-A)^{\frac{1}{2}}\phi(s_{\pi}^-)\rangle\ ds\\
& + \int_{0}^{T}b(u^{\nu}(s),\phi(s_{\pi}^-),u^{\nu}(s))\ ds\\ &+\int_{0}^{T}\langle u^{\nu}(s_{\pi}^+),\partial_{s}\phi(s)\rangle\ ds-\nu^{\frac{1}{2}}\sum_{k=1}^N%
\int_{0}^{T}\langle \sigma_k W^{k}_{s_{\pi}^+},\partial_s\phi(s)\rangle\ ds\\ &+\nu^{\frac{1}{2}} \sum_{k=1}^N\left(\langle \sigma_k W^k_{T},\phi(T) \rangle-\langle \sigma_k W^k_{0},\phi(0) \rangle\right).
\end{align*}
where $s_{\pi}^{-}(s)=t_{i}$ if
$s\in[t_{i},t_{i+1}]$ and $s_{\pi}^{+}(s)=t_{i+1}$ if $s\in[t_{i},t_{i+1}]$.
Taking the limit over a sequence of partitions $\pi_N$ with size going to zero, we get
\begin{align*}
\langle u^{\nu}(T),\phi(T)\rangle& =\langle u^{\nu}(0),\phi(0)\rangle-\int_{0%
}^{T}\langle (-A)^{\frac{1}{2}}u^{\nu}(s),(-A)^{\frac{1}{2}}\phi(s)\rangle\ ds\\
& + \int_{0}^{T}b(u^{\nu}(s),\phi(s),u^{\nu}(s))\ ds\\ &+\int_{0}^{T}\langle u^{\nu}(s),\partial_{s}\phi(s)\rangle\ ds-\nu^{\frac{1}{2}}\sum_{k=1}^N%
\int_{0}^{T}\langle \sigma_k W^{k}_{s},\partial_s\phi(s)\rangle\ ds\\ &+\nu^{\frac{1}{2}} \sum_{k=1}^N\left(\langle \sigma_k W^k_{T},\phi(T) \rangle-\langle \sigma_k W^k_{0},\phi(0) \rangle\right).
\end{align*}
(thanks to the regularity of $u,\ \phi$, dominated convergence theorem and Itô isometry). The argument applies to a generic $t\in[0,T]$, hence we have the thesis.
\end{proof}

Let us now consider the Euler equations
\begin{align}
    \begin{cases}
        \partial_t \bar{u}+\nabla \bar{u}\cdot \bar{u}+\nabla p=\bar{f} \ (x,t)\in D\times(0,T)\\ 
        \operatorname{div}\bar{u}=0 \\ 
        \bar{u}\cdot n|_{\partial D}=0 \\ 
        \bar{u}(0)=\bar{u}_0
    \end{cases}
\end{align}\label{Euler}
\begin{definition}\label{Well Posed Euler }
Given $\bar{u}_0\in H,\ f\in L^2(0,T;H)$ we say that $\bar{u}\in C(0,T;H)$ is a weak solution of equation (\ref{Euler}) if for every $\phi \in C([0,T]; V)\cap C^1([0,T];H)$ \begin{align*} \langle \bar{u}(t),\phi(t)\rangle=\langle \bar{u}_0,\phi (0)\rangle+\int_0^t\langle \bar{u}(s),\partial_s\phi(s)\rangle\ ds +\int_0^t b(\bar{u}(s),\phi(s),\bar{u}(s))\ ds+\int_0^t\langle \bar{f}(s),\phi(s)\rangle\ ds
\end{align*}
for every $t \in[0,T]$ and the energy inequality \begin{align*}
    \lVert{\bar{u}(t)}\rVert^2\leq\lVert{\bar{u}_0}\rVert^2+2\int_0^t \langle f(s),\bar{u}(s)\rangle \ ds 
\end{align*}
holds.
\end{definition}
For what concern the well posedness of the Euler equations the following results hold true, see \cite{kato1967classical}, \cite{lions1996mathematical}.
\begin{theorem}\label{Kato classiche}
If $\bar{u}_0\in C^{1,\epsilon}(\bar{D})\cap H$ and $\bar{f}\in C^{1,\epsilon}([0,T]\times\bar{D})\cap L^2(0,T;H)$, then there exist $\bar{u}, \bar{p}$ classical solutions of equation (\ref{Euler}). Moreover, $\bar{u},\  \nabla\bar{u},\ p,\ \nabla p\in C([0,T]\times\bar{D})$,\  $\bar{u}$ is unique and $p$ is unique up to an arbitrary function of $t$ which can be added to $p$.
\end{theorem}
\begin{theorem}\label{P.L. Lions}
If $f=0$, $u$ is a weak solution of the Euler equations with initial condition $u_0\in H$ and $\Bar{u}$ is the unique weak solution of the Euler equations with initial condition $\Bar{u}_0\in H\cap C^{1,\epsilon}(\bar{D})$, then \begin{align*}
    \lVert (u-\bar{u})(t)\rVert^2\leq e^{2t\lVert\nabla \bar{ u}\rVert_{L^{\infty}([0,T]\times D)}}\lVert u_0- \bar{u}_0\rVert^2.\end{align*}
Calling
    $$O_n =\left\{u_0\in H: \ \exists \bar{u_0}\in H\cap C^{1,\epsilon}(\bar{D}),\ \lVert u_0-\bar{u}_0\rVert <\frac{1}{n}e^{-3T\lVert\nabla \bar{ u}\rVert_{L^{\infty}([0,T]\times D)}}\right\}$$ 
where $\bar{u}$ is the solution of the Euler equations with initial condition $\Bar{u}_0$, then for each $u_0\in \bigcap_{n\geq 1}O_n=:\Tilde{O} $ there exists a unique $u\in C([0,T],H)$ weak solution of the Euler equations with initial condition $u_0$.
Moreover the energy equality \begin{align*}
    \lVert u(t)\rVert^2=\lVert u_0\rVert^2
\end{align*}
holds.
\end{theorem}
For each $u_0\in \Tilde{O}$ we will say that $\{\bar{u}^m_0\}_{m\in \mathbb{N}}$ approximates $u_0$ in the sense of Theorem \ref{P.L. Lions} if $\bar{u}^m_0\in H\cap C^{1,\epsilon}(\bar{D})$ and $$\lVert u_0-\bar{u}^m_0\rVert <\frac{1}{m}e^{-3T\lVert\nabla \bar{ u}^m\rVert_{L^{\infty}([0,T]\times D)}}$$ 
where $\bar{u}^m$ is the solution of the Euler equations with initial condition $\Bar{u}_0^m$.\\ 
Lastly we introduce some results related to the boundary layer corrector of the solution of the Euler equations, see \cite{kato1984remarks}.
\begin{proposition}
Under the assumptions of Theorem \ref{Kato classiche}:
\begin{itemize}
    \item it exists a smooth skew-symmetric matrix  $a$ such that $\bar{u}=\operatorname{div}a $ on $\partial D$ and $a=0$ on $\partial D$;
    \item   let  $\xi: \mathbb{R}^{+}\rightarrow \mathbb{R}^+$ a smooth function such that $\xi(0)=1,\  \xi(r)=0 $ if $r\geq 1$ and $$z:D\rightarrow \mathbb{R}^{+},\ z(x)=\xi (\rho/\delta) \  \ with \ \ \rho=dist(x,\partial D)$$ and $\delta$ a parameter which goes to 0 when $\nu$ goes to $0$. Let, moreover,  $v=\operatorname{div}(za)$. Then, $$\bar{u}-v\in C([0,T]; V)\cap C^1([0,T];H).$$
$\operatorname{supp}(v)$ is the boundary layer of width $\delta$ that we denote by $\Gamma_{\delta}$; 
\item the following estimates hold true
$$\lVert v(t)\rVert_{L^{\infty}(D)}\leq K, \ \lVert v(t)\rVert_{L^2(D)}\leq K\delta^{\frac{1}{2}}, \ \lVert\partial_t v(t)\rVert_{L^2(D)}\leq K\delta^{\frac{1}{2}},$$  $$\lVert\nabla v(t)\rVert_{L^{\infty}(D)}\leq K\delta^{-1}, 
\  \lVert\nabla v(t)\rVert_{L^2(D)}\leq K\delta^{-1/2},  \  \lVert\rho(t)\nabla v(t)\rVert_{L^{\infty}(D)}\leq K,$$ $$  \lVert\rho(t)^2 \nabla v(t)\rVert_{L^{\infty}(D)}\leq K\delta, \  \lVert\rho(t) \nabla v(t)\rVert_{L^2(D)}\leq K\delta^{\frac{1}{2}}.$$ where the coefficient $K$ depends from $\bar{u}$ and it is independent from $t$.
\end{itemize}
\end{proposition}
Now we can state our main theorems. Since the stochastic term in equation (\ref{additive noise}) goes to $0$, we assume that the external force in the Euler equations is identically $0$. Theorem \ref{Thm strong strong} is a generalization of the results in \cite{kato1984remarks} to this stochastic framework and also the idea of the proof is similar. In Theorem \ref{Theorem weak strong} we consider a wilder set of initial conditions. 

\begin{theorem}\label{Thm strong strong}
If $\bar{u}_0\in C^{1,\epsilon}(\bar{D})$ and under previous assumptions on $u_0^{\nu}$ and $\sigma_k$, if $$\lim_{\nu\rightarrow 0}\mathbb{E}\left[\lVert u^{\nu}_0-\bar{u}_0\rVert^2\right]=0, $$ then the following are equivalent:
\begin{enumerate}
    \item $\lim_{\nu \rightarrow 0}\mathbb{E}\left[\sup_{t\in [0,T]}\lVert u^{\nu}-\bar{u}\rVert^2\right]= 0.$
    \item $u^{\nu}(t)\rightharpoonup\bar{u}(t) $ in $L^2(\Omega\times D)$ for each $t\in [0,T]$.
    \item $\lim_{\nu\rightarrow 0}\nu\int_0^T \mathbb{E}\left[\lVert\nabla u^{\nu}(t) \rVert^2_{L^2(D)}\right] \ dt= 0.$
    \item $\lim_{\nu\rightarrow 0}\nu\int_0^T \mathbb{E}\left[\lVert\nabla u^{\nu}(t) \rVert^2_{L^2(\Gamma_{c\nu})}\right] \ dt= 0.$
\end{enumerate}
\end{theorem}

\begin{theorem}\label{Theorem weak strong}
If $u_0\in \Tilde{O}$, $u_0^n\in L^2(\mathcal{F}_0,H),\ \lim_{n\rightarrow+\infty}\mathbb{E}\left[\lVert u_0^n-u_0\rVert^2\right]=0$. Let $u$ be the solution of the Euler equations with initial condition $u_0$, $u^n$ be the solution of the stochastic Navier-Stokes equations with viscosity $\nu_n$ and initial condition $u^n_0$. If $$\lim_{n\rightarrow+\infty }\nu_n=0,\ \ \lim_{n\rightarrow +\infty}\nu_n\int_0^T \mathbb{E}\left[\lVert\nabla u^{n}(t) \rVert^2_{L^2(\Gamma_{c\nu_n})}\right] \ dt= 0,  $$ then 
$$\lim_{n \rightarrow +\infty}\mathbb{E}\left[\sup_{t\in [0,T]}\lVert u^{n}-u\rVert^2\right]= 0. $$
\end{theorem}
\begin{remark}

Theorem \ref{Thm strong strong} means that if convergence does not take place, the energy dissipation within the boundary layer of width $c\nu$ must remain finite as $\nu \rightarrow 0$. This suggests that something violent must have happened.
\end{remark}
\begin{remark}\label{remark novelty}
Theorem \ref{Theorem weak strong} is new also in the deterministic framework, namely taking $\sigma_k=0\ \forall \ k=1,\dots, n$. In section \ref{Sec Rem} we will prove this result in the deterministic framework for a non-zero external force.
\end{remark}
\begin{remark}
$K$ will denote several constants dependent only from the solution of the Euler equations and its data, $\{ \sigma_k \}_{k=1}^N$ and $T$ in the following.
\end{remark}
\section{Proof of Theorem \ref{Thm strong strong}}\label{Sec Thm strong strong }
The proof of theorem \ref{Thm strong strong} follows from a preliminary weaker result, namely under the same assumptions $$\lim_{\nu \rightarrow 0}\sup_{t\in [0,T]}\mathbb{E}[\lVert u^{\nu}-u\rVert_H^2]= 0. $$ This is the analogous of the Kato's result in this stochastic framework.

\begin{proposition}\label{weak result strong assump}
Under the same assumptions of Theorem \ref{Thm strong strong}, if $$\lim_{\nu\rightarrow 0}\mathbb{E}\left[\lVert u^{\nu}_0-\bar{u}_0\rVert^2\right]=0, $$ then the following are equivalent:
\begin{enumerate}
    \item $\lim_{\nu \rightarrow 0}\sup_{t\in [0,T]}\mathbb{E}\left[\lVert u^{\nu}-\bar{u}\rVert^2\right]= 0.$
    \item $u^{\nu}(t)\rightharpoonup\bar{u}(t) $ in $L^2(\Omega\times D)$ for each $t\in [0,T]$.
    \item $\lim_{\nu\rightarrow 0}\nu\int_0^T \mathbb{E}\left[\lVert\nabla u^{\nu}(t) \rVert^2_{L^2(D)}\right] \ dt= 0.$
    \item $\lim_{\nu\rightarrow 0}\nu\int_0^T \mathbb{E}\left[\lVert\nabla u^{\nu}(t) \rVert^2_{L^2(\Gamma_{c\nu})}\right] \ dt= 0.$
\end{enumerate}
\end{proposition}
\begin{proof}
$1.\Rightarrow 2.$ and $3. \Rightarrow 4.$ are obvious. We need only prove that 2. $\Rightarrow$ 3. and 4. $\Rightarrow$ 1.\begin{itemize}
    \item[2. $\Rightarrow$ 3.] By energy equality for each $t=T$
    \begin{align*}
        \nu\mathbb{E}\left[\int_0^T \lVert \nabla u^{\nu}(s)\rVert_{L^2(D)}^2\right]=\frac{1}{2}\mathbb{E}\left[\lVert u_0^{\nu}\rVert^2\right]-\frac{1}{2}\mathbb{E}\left[\lVert u^{\nu}(T)\rVert^2\right]+T\nu\sum_{k=1}^N \lVert\sigma_k\rVert^2.
    \end{align*}  Taking the limsup of this expression and exploiting the fact that under the assumptions \begin{align*}
        \mathbb{E}\left[\lVert u^{\nu}_0-\bar{u}_0\rVert^2\right]\rightarrow 0
    \end{align*}
    \begin{align*}
        \lVert \bar{u}(T)\rVert^2\leq\liminf_{\nu\rightarrow 0}\mathbb{E}\left[\lVert u^{\nu}(T)\rVert^2\right]
    \end{align*} we get the thesis.
    \item[4. $\Rightarrow$ 1.] For each time $t$ we have
\begin{align*}
     \mathbb{E}[\lVert u^{\nu}-\bar{u}\rVert^2]&=\mathbb{E}[\lVert u^{\nu}\rVert^2]+\lVert \bar{u}\rVert ^2-2\mathbb{E}[\langle u^{\nu},\bar{u}\rangle]\\ & \stackrel{energy \ eq.}{\leq}  \mathbb{E}[\lVert u_0^{\nu}\rVert^2]+t\nu\sum_{k}\lVert \sigma_k\rVert^2+\lVert \bar{u}_0\rVert^2-2\mathbb{E}[\langle u^{\nu},\bar{u}\rangle] \\ & \stackrel{\mathbb{E}[\lVert u_0^{\nu}-\bar{u}_0\rVert_{L^2(D)}^2]\rightarrow 0}{\leq} 
    o(1)+2\lVert \bar{u}_0\rVert^2+t\nu\sum_{k}\lVert \sigma_k\rVert^2-2\mathbb{E}[\langle u^{\nu},\bar{u}\rangle]\\ & = o(1)+2\lVert \bar{u}_0\rVert^2+t\nu\sum_{k}\lVert\sigma_k\rVert^2-2\mathbb{E}[\langle u^{\nu},\bar{u}-v\rangle]-2\mathbb{E}[\langle u^{\nu},v\rangle]
\end{align*}
Then \begin{align}
\mathbb{E}[\lVert u^{\nu}-\bar{u}\rVert^2]\leq  o(1)+2\lVert \bar{u}_0\rVert^2+t\nu\sum_{k}\lVert\sigma_k\rVert^2-2\mathbb{E}[\langle u^{\nu},\bar{u}-v\rangle]-2\mathbb{E}[\langle u^{\nu},v\rangle]
\label{first estimate first prop}\end{align}

To analyze the second-last term we use the weak formulation of $u^{\nu}$ for the test function $\bar{u}-v$.
\begin{align*} & -2\langle u^{\nu}(t),(\bar{u}-v)(t)\rangle =-2\langle u^{\nu}(0),(\bar{u}-v) (0)\rangle-2\int_0^t\langle u^{\nu}(s),\partial_s(\bar{u}-v)(s)\rangle\ ds+\\ & 2\nu\int_0^t \langle(-A)^{\frac{1}{2}}u^{\nu}(s), (-A)^{\frac{1}{2}}(\bar{u}-v)(s)\rangle\ ds- \int_0^t 2b(u^{\nu}(s),(\bar{u}-v)(s),u^{\nu}(s))\ ds-\\ & 2\nu^{\frac{1}{2}}\sum_{k}\langle \sigma_k, (\bar{u}-v)(t)\rangle W^k_t+ 2\nu^{\frac{1}{2}}\sum_{k=1}^N\int_0^t \langle \sigma_k, (\bar{u}-v)(s)\rangle W^k_s \ ds.
\end{align*} 
Taking the expected value of the last expression we obtain
\begin{align*}&  -2\mathbb{E}\left[\langle u^{\nu}(t),(\bar{u}-v)(t)\rangle\right]+2\mathbb{E}\left[\lVert u^{\nu}_0\rVert^2\right]\stackrel{\mathbb{E}\left[\lVert u_0^{\nu}-\bar{u}_0\rVert^2\right]\rightarrow 0, \ \lVert v(t)\rVert _{L^2(D)}\leq K\delta^{\frac{1}{2}} }{=}\\ & o(1)  -2\mathbb{E}\left[\int_0^t\langle u^{\nu}(s),\partial_s(\bar{u}-v)(s)\rangle\ ds\right]+  2\nu\mathbb{E}\left[\int_0^t \langle(-A)^{\frac{1}{2}}u^{\nu}(s), (-A)^{\frac{1}{2}}(\bar{u}-v)(s)\rangle\ ds\right]\\ & - \mathbb{E}\left[\int_0^t 2b(u^{\nu}(s),(\bar{u}-v)(s),u^{\nu}(s))\ ds\right].
\end{align*}
Moreover
\begin{align*}
    -\mathbb{E}\left[\langle u^{\nu}(s),\partial_s(\bar{u}-v)(s)\rangle\right] & \stackrel{energy \ eq. \ \lVert \partial_t v(t)\rVert_{L^2(D)}\leq K\delta^{\frac{1}{2}}}{=}o(1)-\mathbb{E}\left[\langle u^{\nu}(s),\partial_s \bar{u}(s)\rangle\right] \\ & \stackrel{Euler\  eq}{=} o(1)+\mathbb{E}\left[\langle u^{\nu}(s), \nabla \bar{u}\cdot \bar{u}(s)\rangle\right],
\end{align*}
 \begin{align*}
      \mathbb{E}\left[\langle u^{\nu}(t),v(t)\rangle\right]\leq \lVert v(t)\rVert\mathbb{E}\left[\lVert u^{\nu}(t)\rVert^2\right]^{\frac{1}{2}}\stackrel{energy \ eq.}{\leq}K\lVert v(t)\rVert\stackrel{\lVert v(t)\rVert_{L^2(D)}\leq K\delta^{\frac{1}{2}}}{\leq} K\delta^{\frac{1}{2}}\rightarrow 0, \end{align*}
 \begin{align*}
    \mathbb{E}\left[\int_0^t b(u^{\nu}(s)-\bar{u}(s),\bar{u}(s),u^{\nu}(s)-\bar{u}(s))\ ds\right]\leq \lVert \nabla \bar{u}\rVert_{L^{\infty}(0,T;D)}\mathbb{E}\left[\int_0^t \lVert (u^{\nu}-\bar{u})(s)\rVert^2\ ds\right]. 
\end{align*} 
 Using all these relations in equation (\ref{first estimate first prop}), we get 
 \begin{align*}
     \mathbb{E}\left[\lVert u^{\nu}-\bar{u}\rVert_{H}^2\right]&\leq o(1)+ 2\lVert \bar{u}_0\rVert^2-2\mathbb{E}\left[\lVert u_0^{\nu}\rVert^2\right]+t\nu\sum_{k=1}^N\lVert\sigma_k\rVert^2+\\&  2\nu\mathbb{E}\left[\int_0^t \langle(-A)^{\frac{1}{2}}u^{\nu}(s), (-A)^{\frac{1}{2}}(\bar{u}-v)(s)\rangle\ ds\right]+2\mathbb{E}\left[\int_0^t b(u^{\nu}(s),v(s),u^{\nu}(s)) \ ds\right]+\\ &
    2\mathbb{E}\left[\int_0^t(b(u^{\nu}(s),\bar{u}(s),\bar{u}(s))-b(u^{\nu}(s),\bar{u}(s),u^{\nu}(s)))\ ds\right]\\ &
    \stackrel{b(u^{\nu},\bar{u},u^{\nu})-b(u^{\nu},\bar{u},\bar{u})=b(u^{\nu}-\bar{u},\bar{u},u^{\nu}-\bar{u})}{=} o(1)+t\nu\sum_{k=1}^N \lVert \sigma_k\rVert^2\\ &+2\nu\mathbb{E}\left[\int_0^t \langle(-A)^{\frac{1}{2}}u^{\nu}(s), (-A)^{\frac{1}{2}}(\bar{u}-v)(s)\rangle\ ds\right]\\ & +2\mathbb{E}\left[\int_0^t b(u^{\nu}(s),v(s),u^{\nu}(s)) \ ds\right]-2\mathbb{E}\left[\int_0^t b(u^{\nu}(s)-\bar{u}(s),\bar{u}(s),u^{\nu}(s)-\bar{u}(s))\ ds\right].
\end{align*}  
 
thus, calling $$R(s)=\nu \sum_{k}\lVert \sigma_k\rVert^2+2\nu\mathbb{E}\left[\langle(-A)^{\frac{1}{2}}u^{\nu}(s), (-A)^{\frac{1}{2}}(\bar{u}-v)(s)\rangle\right]+\mathbb{E}\left[b(u^{\nu}(s),v(s),u^{\nu}(s))\right]$$ we have
\begin{align*}
    \mathbb{E}\left[\lVert u^{\nu}-\bar{u}\rVert_{H}^2\right]\leq o(1)+\int_0^t(K\mathbb{E}\left[\lVert(u^{\nu}-\bar{u})(s)\rVert^2\right]+R(s))\ ds.
\end{align*}
If we are able to prove that $\lim_{\nu\rightarrow 0}\int_0^t R(s) \ ds=0 $, then via Gronwall's inequality we'll get the thesis. The term related to $\sigma_k$ is obvious. For what concerns the others:
\begin{align*} 
  \left|\mathbb{E}\left[ b(u^{\nu}(s),v(s),u^{\nu}(s))\right]\right|& \leq \mathbb{E}\left[\int_{\Gamma_{\delta}} \lvert\nabla v\rvert(s)\rho^2  \frac{\lvert u^{\nu}(s)\rvert^2}{\rho^2}\ dx \right]\\ &
    \stackrel{\lVert\rho^2 \nabla v(t)\rVert_{L^{\infty}(D)}\leq K\delta}{\leq} K\delta\mathbb{E}\left[\lVert\frac{u^{\nu}}{\rho}\rVert^2_{L^2(\Gamma_{\delta})} \right]\\ &\stackrel{Hardy-Littlewood \ ineq.}{\leq} K\delta \mathbb{E}\left [ \lVert\nabla u^{\nu}\rVert^2_{L^2(\Gamma_{\delta})}\right],
\end{align*}
\begin{align*}
     \left|\mathbb{E}\left[\nu\langle(-A)^{\frac{1}{2}}u^{\nu}(s), (-A)^{\frac{1}{2}}(\bar{u}-v)(s)\rangle\right]\right| & \leq \nu \mathbb{E}\left[\lVert \nabla u^{\nu}(s)\rVert_{L^2(D)}\lVert\nabla \bar{u}(s)\rVert_{L^2(D)}\right] +\nu \mathbb{E}\left[\lVert \nabla u^{\nu}(s)\rVert_{L^2(\Gamma_{\delta})}\lVert\nabla v(s)\rVert_{L^2(\Gamma_{\delta})}\right]\\
     & \stackrel{\lVert\nabla v(t)\rVert_{L^{2}(D)}\leq K\delta^{-1/2}}{\leq} K\nu \mathbb{E}\left[\lVert\nabla u^{\nu}(s)\rVert_{L^2(D)}\right]+\nu K \delta^{-1/2}\mathbb{E}\left[\lVert\nabla u^{\nu}(s)\rVert_{L^2(\Gamma_{\delta})}\right].
\end{align*}
Taking $\delta=c\nu$, we have 
\begin{align*}
    R(t) \leq  K\nu+K\nu \mathbb{E}[\lVert\nabla u^{\nu}(s)\rVert_{L^2(D)}]+\nu^{\frac{1}{2}} K \mathbb{E}\left[\lVert\nabla u^{\nu}(s)\rVert_{L^2(\Gamma_{\delta})}\right]+K\nu \mathbb{E}\left [ \lVert \nabla u^{\nu}(s)\rVert^2_{L^2(\Gamma_{\delta})}\right]. 
\end{align*}
Exploiting the assumption $$\lim_{\nu\rightarrow 0}\nu\int_0^T \mathbb{E}\left[\lVert\nabla u^{\nu}(t) \rVert^2_{L^2(\Gamma_{c\nu})}\right] \ dt= 0,$$ previous estimates and energy equality, via Holder inequality we get $\lim_{\nu\rightarrow 0}\int_0^t R(s) \ ds=0 $ and then the thesis.
\end{itemize}
\end{proof}
\begin{corollary}\label{Strong result strong assum}
Under the same assumptions of Proposition \ref{weak result strong assump}, if $$\lim_{\nu\rightarrow 0}\nu\int_0^T \mathbb{E}\left[\lVert\nabla u^{\nu}(t) \rVert^2_{L^2(\Gamma_{c\nu})}\right] \ dt= 0, $$ then $$\lim_{\nu \rightarrow 0}\mathbb{E}\left[\sup_{t\in [0,T]}\lVert u^{\nu}-\bar{u}\rVert^2\right]= 0. $$
\end{corollary}

\begin{proof}
Preliminarily, note that, starting from equation (\ref{uniform energy}), we have
\begin{align*}
  \mathbb{E}\left[sup_{t\in [0,T]}\lVert u^{\nu}(t)\rVert^2\right]& \leq \mathbb{E}\left[\lVert u_0^{\nu}\rVert^2\right]+T\nu^{\frac{1}{2}} \sum_{k=1}^N \lVert\sigma_k\rVert^2+K\nu^{\frac{1}{2}} \sum_{k=1}^N\mathbb{E}\left[\int_0^T\langle u^{\nu}(s),\sigma_k\rangle^2 \ ds\right]^{\frac{1}{2}} \\ & \leq
  \mathbb{E}\left[\lVert u_0^{\nu}\rVert^2\right]+T\nu^{\frac{1}{2}} \sum_{k=1}^N \lVert\sigma_k\rVert^2+K\nu^{\frac{1}{2}}\sum_{k=1}^N\lVert\sigma_k\rVert\mathbb{E}\left[\int_0^T\lVert u^{\nu}(s)\rVert^2\ ds\right]^{\frac{1}{2}}\\
  & \leq  \mathbb{E}\left[\lVert u_0^{\nu}\rVert^2\right]+T\nu^{\frac{1}{2}} \sum_{k=1}^N \lVert\sigma_k\rVert^2+K\nu^{\frac{1}{2}}\sum_{k=1}^N\lVert\sigma_k\rVert\left(\int_0^T \mathbb{E}\left[\lVert u_0^{\nu}\rVert^2\right]+s\nu \sum_{j=1}^N \lVert\sigma_j\rVert^2 \ ds\right)^{\frac{1}{2}}\\
  & \leq K<+\infty.
\end{align*}
Now the proof is similar to the previous one. For each time $t$ we have
\begin{align*}
    \lVert u^{\nu}-\bar{u}\rVert^2 & =\lVert u^{\nu}\rVert^2+\lVert \bar{u}\rVert^2-2\langle u^{\nu},\bar{u}\rangle\\
    & \stackrel{It\hat{o}\  formula}{\leq}\lVert u_0^{\nu}\rVert^2+t\nu\sum_{k=1}^N\lVert\sigma_k\rVert^2+2\sum_{k=1}^N\nu^{\frac{1}{2}}\int_0^t\langle u^{\nu}(s),\sigma_k\rangle dW^k_s+\lVert \bar{u}_0\rVert^2-2\langle u^{\nu},\bar{u}\rangle \\ & =\lVert u_0^{\nu}\rVert^2+t\nu\sum_{k=1}^N\lVert\sigma_k\rVert^2+2\sum_{k=1}^N\nu^{\frac{1}{2}}\int_0^t\langle u^{\nu}(s),\sigma_k\rangle dW^k_s+\lVert \bar{u}_0\rVert^2-2\langle u^{\nu},\bar{u}-v\rangle_H-2\langle u^{\nu},v\rangle.
\end{align*}
Let us rewrite $\langle u^{\nu},\bar{u}-v\rangle$ thanks to the weak formulation of $u^{\nu}$
\begin{align*}
    -2\langle u^{\nu},\bar{u}-v\rangle & =-2\langle u^{\nu}_0,(\bar{u}-v)(0)\rangle-2\int_0^t\langle u^{\nu}(s),\partial_s(\bar{u}-v)(s)\rangle\ ds\\ &+ 2\nu\int_0^t\langle(-A)^{\frac{1}{2}}u^{\nu}(s),(-A)^{\frac{1}{2}}(\bar{u}-v)(s)\rangle\ ds-2\int_0^t b(u^{\nu},\bar{u}-v,u^{\nu})(s)\ ds\\
    & -2\nu^{\frac{1}{2}}\sum_{k=1}^N\langle\sigma_k,(\bar{u}-v)\rangle W^k_t+2\nu^{\frac{1}{2}}\sum_{k=1}^N\int_0^t\langle\sigma_k,(\bar{u}-v)(s)\rangle W^k_s\ ds.
\end{align*}
Moreover, thanks to previous relation and  \begin{align*}
-2\langle u^{\nu}(s),\partial_s(\bar{u}-v)(s)\rangle=2\langle u^{\nu}(s),\partial_s v(s)\rangle+2\langle u^{\nu}(s),\nabla \bar{u}\cdot \bar{u}(s)\rangle,
\end{align*}
\begin{align*}
  b(u^{\nu},\bar{u},u^{\nu})-b(u^{\nu},\bar{u},\bar{u})=b(u^{\nu}-\bar{u},\bar{u},u^{\nu}-\bar{u}),  
\end{align*}
 we have at time $t$ 
\begin{align*}
    \lVert u^{\nu}-\bar{u}\rVert^2& =\left(\lVert u_0^{\nu}\rVert^2+\lVert \bar{u}_0\rVert^2-2\langle u^{\nu}_0,(\bar{u}-v)(0)\rangle \right)+(t\nu\sum_{k=1}^N\lVert\sigma_k\rVert^2+2\nu^{\frac{1}{2}}\sum_{k=1}^N\int_0^t\langle u^{\nu}(s),\sigma_k\rangle dW^k_s\\ &-2\nu^{\frac{1}{2}}\sum_{k=1}^N\langle\sigma_k,(\bar{u}-v)(t)\rangle W^k_t+2\nu^{\frac{1}{2}}\sum_{k=1}^N\int_0^t \langle\sigma_k,(\bar{u}-v)(s)\rangle W^k_s \ ds)\\ &+\left(2\int_0^t b(u^{\nu},v,u^{\nu})(s) \ ds-2\int_0^t b(u^{\nu}-\bar{u},\bar{u},u^{\nu}-\bar{u})(s)\ ds\right)+\\ & (-2\langle u^{\nu},v\rangle+2\nu \int_0^t\langle(-A)^{\frac{1}{2}}u^{\nu}(s),(-A)^{\frac{1}{2}}(\bar{u}-v)(s)\rangle\ ds+2\int_0^t \langle u^{\nu},\partial_s v\rangle\ ds)\\ &= I_1(t)+I_2(t)+I_3(t)+I_4(t).
\end{align*}
Thus $$\mathbb{E}\left[ sup_{t\in [0,T]}\lVert u^{\nu}-u\rVert^2\right]\leq \mathbb{E}[sup_{t\in [0,T]}I_1]+\mathbb{E}[sup_{t\in [0,T]}I_2]+\mathbb{E}[sup_{t\in [0,T]}I_3]+\mathbb{E}[sup_{t\in [0,T]}I_4]$$
\begin{align*}
    E\left[\sup_{t\in[0,T]} I_1\right]=& \mathbb{E}\left[\lVert u_0^{\nu}\rVert^2+\lVert \bar{u}_0\rVert^2-2\langle u^{\nu}_0,(\bar{u}-v)(0)\rangle\right]\\ \leq & -\mathbb{E}\left[\lVert u_0^{\nu}\rVert^2\right]+\lVert \bar{u}_0\rVert^2+2\mathbb{E}\left[\lVert u_0^{\nu}\rVert\lVert u_0^{\nu}-\bar{u}_0\rVert\right]+2\mathbb{E}\left[\lVert\bar{u}_0^{\nu}\rvert\lVert v(0)\rVert\right]\\
    & \stackrel{\lVert v(t)\lVert_{L^2(D)}\leq K\delta^{\frac{1}{2}},\ \mathbb{E}\left[\lVert u_0^{\nu}-\bar{u}_0\rVert^2\right]\rightarrow 0}{=}o(1)+K\delta^{\frac{1}{2}}.
\end{align*}
The analysis of $I_3$ is similar to what we have done in the previous proposition, hence some details have been omitted.
\begin{align*}
    \mathbb{E}\left[sup_{t\in [0,T]} I_3\right] &\leq 2 \mathbb{E}\left[sup_{t\in [0,T]}\int_0^t b(u^{\nu},v,u^{\nu})(s)\ ds\right]+2 \mathbb{E}\left[sup_{t\in[0,T]}\int_0^t b(u^{\nu}-\bar{u},\bar{u},u^{\nu}-\bar{u})(s)\ ds\right]\\
    &\leq 2 \mathbb{E}\left[\int_0^T|b(u^{\nu},v,u^{\nu})(s)|\ ds\right]+2 \mathbb{E}\left[\int_0^T|b(u^{\nu}-\bar{u},\bar{u},u^{\nu}-\bar{u})(s)|\ ds\right]\\
    & \leq\mathbb{E}\left[\int_0^T \lVert\rho^2\nabla v(s)\rVert_{L^{\infty}(D)}\lVert u^{\nu}\rho^{-1}(s)\rVert_{L^2(D)}^2\right]+\mathbb{E}\left[\int_0^T\lVert u^{\nu}-\bar{u}\rVert^2(s)\lVert \bar{u}(s)\rVert_{L^{\infty}(D)}\ ds\right]\\ &\leq 2K\delta \mathbb{E}\left[\int_0^T\lVert\nabla u^{\nu}(s)\rVert_{L^2(\Gamma_{\delta})}^2\ ds\right]+K\mathbb{E}\left[\int_0^T\lVert u^{\nu}-\bar{u}\rVert^2(s)\ ds \right].
\end{align*}
The last term goes to 0 thanks to previous proposition, the first term goes to $0$ thanks to the assumptions choosing $\delta$ properly. More details will follow. \\ 
Let us analyze all the elements of $I_2$ exploiting previous energy equalities and properties of Brownian motion:
\begin{align*}
    \mathbb{E}\left[\sup_{t\in[0,T]} T\nu \sum_{k=1}^N \lVert\sigma_k\rVert^2\right]\leq K\nu
\end{align*}
\begin{align*}
    \mathbb{E}\left[sup_{t\in [0,T]} 2\sum_{k=1}^N \nu^{\frac{1}{2}}\int_0^t \langle u^{\nu},\sigma_k\rangle dW^k_s\right]\stackrel{Doob's\  ineq}{\leq} K\nu^{\frac{1}{2}}
\end{align*}
\begin{align*}
    \mathbb{E}\left[sup_{t\in [0,T]} 2\nu^{\frac{1}{2}}\sum_{k=1}^N\langle\sigma_k,\bar{u}-v(t)\rangle W^k_t\right]\leq K\nu^{\frac{1}{2}}\mathbb{E}\left[sup_{t\in [0,T]}|W^k_t|\right]\leq K\nu^{\frac{1}{2}}
\end{align*}
\begin{align*}
    \mathbb{E}\left[\sup_{t\in[0,T]} 2\nu^{\frac{1}{2}}\sum_{k=1}^N \int_0^t \langle\sigma_k,\bar{u}-v(s)\rangle W^k_s\ ds\right]\leq K\nu^{\frac{1}{2}}.
\end{align*}
It remains only to analyze $I_4$. Some of the estimates below use tricks already presented, hence some details have been omitted.
\begin{align*}
    2\nu \mathbb{E}\left[\sup_{t\in [0,T]}\int_0^t \langle(-A)^{\frac{1}{2}}u^{\nu},(-A)^{\frac{1}{2}}(\bar{u}-v)\rangle\ ds\right]\leq 2\nu \mathbb{E}\left[\int_0^T K\lVert \nabla u^{\nu}(s)\rVert_{L^2(D)}+K\delta^{-\frac{1}{2}}\lVert\nabla u^{\nu}(s)\rVert_{L^2(\Gamma_{\delta})}\ ds \right]
\end{align*}
\begin{align*}
    \mathbb{E}\left[\sup_{t\in[0,T]} \langle u^{\nu},v\rangle\right]\leq \mathbb{E}\left[\sup_{t\in[0,T]} \lVert u^{\nu}\rVert\sup_{t\in[0,T]}\lVert v\rVert \right]\leq \delta^{\frac{1}{2}}\mathbb{E}\left[\sup_{t\in[0,T]} \lVert u^{\nu}\rVert^2\right]^{\frac{1}{2}}\leq K\delta^{\frac{1}{2}}
\end{align*}
\begin{align*}
    \mathbb{E}\left[sup_{t\in[0,T]} \int_0^t \langle u^{\nu},\partial_s v\rangle\ ds  \right]\leq \lVert \partial_s v\rVert_{L^{\infty}(0,T;L^2(D))}\mathbb{E}\left[\int_0^T\lVert u^{\nu}(s)\rvert\ ds \right]\leq K \delta^{\frac{1}{2}}.
\end{align*}
In conclusion, if we take $\delta=c\nu$, then \begin{align*}
    \mathbb{E}\left[\sup_{t\in [0,T]}\lVert u^{\nu}-u\rVert^2\right] & \leq o(1)+K\nu^{\frac{1}{2}}+2K\nu \mathbb{E}\left[\int_0^T\lVert\nabla u^{\nu}(s)\rVert_{L^2(\Gamma_{c\nu})}^2\ ds\right]+K\mathbb{E}\left[\int_0^T\lVert u^{\nu}-\bar{u}(s)\rVert^2\ ds \right]\\ &+K\nu +\nu K \mathbb{E}\left[\int_0^T \lVert \nabla u^{\nu}(s)\rVert_{L^2(D)}\ ds\right]+K\nu^{\frac{1}{2}}\mathbb{E}\left[\int_0^T \lVert\nabla u^{\nu}\rVert_{L^2(\Gamma_{c\nu})}\ ds \right].
\end{align*}
It is clear the almost all the terms goes to $0$ thanks to the assumptions and Proposition \ref{weak result strong assump}. We need just to check that $\nu  \mathbb{E}\left[\int_0^T \lVert \nabla u^{\nu}(s)\rVert_{L^2(D)}\ ds\right]$ and $\nu^{\frac{1}{2}}\mathbb{E}\left[\int_0^T \lVert\nabla u^{\nu}\rVert_{L^2(\Gamma_{c\nu})}\ ds \right]$ behave properly, but this is elementary, in fact:
\begin{align*}
\nu  \mathbb{E}\left[\int_0^T \lVert\nabla u^{\nu}(s)\rVert_{L^2(D)}\ ds\right]\leq \nu T\mathbb{E}\left[\int_0^T \lVert\nabla u^{\nu}(s)\rVert_{L^2(D)}^2 ds\right]^{\frac{1}{2}}=\nu^{\frac{1}{2}}T\mathbb{E}\left[\nu\int_0^T \lVert \nabla u^{\nu}(s)\rVert_{L^2(D)}^2 ds\right]^{\frac{1}{2}}\leq K\nu^{\frac{1}{2}}    
\end{align*}
\begin{align*}
\nu^{\frac{1}{2}}\mathbb{E}\left[\int_0^T \lVert\nabla u^{\nu}\rVert_{L^2(\Gamma_{c\nu})}\ ds \right]\leq K\mathbb{E}\left[\nu \int_0^T\lVert \nabla u^{\nu}\rVert_{L^2(\Gamma_{c\nu})}^2\ ds\right]^{\frac{1}{2}}\stackrel{\nu\rightarrow 0}{\rightarrow} 0.   
\end{align*}
This completes the proof.
\end{proof}\\
Theorem \ref{Thm strong strong} follows immediately by Proposition \ref{weak result strong assump} and Corollary \ref{Strong result strong assum}.
\section{Proof of Theorem \ref{Theorem weak strong}} \label{Sec Theorem weak strong}
As in the previous section, we start with a weaker result with the supremum in time outside the expected value to obtain the stronger one with the supremum in time inside the expected value. The idea behind both the proofs is simply to introduce an approximation of $u_0$ in the sense of Theorem \ref{P.L. Lions}, then $$\lVert u^n-u\rVert^2\leq 2\lVert u^n-\bar{u}^m\rVert^2 +2\lVert \bar{u}^m-u\rVert^2,  $$ where $\bar{u}^m$ is the solution of the Euler Equations with initial condition $\bar{u}^m_0\in H\cap C^{1+\epsilon}(\bar{D})$. Thus, the second term can be estimate via Theorem \ref{P.L. Lions}, the first one is analyzed exploiting techniques similar to the ones of the previous section.
\begin{remark}\label{Remark dis}
If $u_0\in \Tilde{O}$ and $\{\bar{u}^m_0\}_{m\in\mathbb{N}}$ approximates $u_0$ in the sense of Theorem \ref{P.L. Lions}, then $$\lVert u_0-\bar{u}_0^m\rVert e^{2T\lVert\nabla \bar{u}^m\rVert_{L^{\infty}([0,T]\times D)}}\leq\frac{1}{m}e^{-T\lVert\nabla \bar{u}^m\rVert_{L^{\infty}([0,T]\times D)}}$$
$$ \lVert u_0-\bar{u}_0^m\rVert e^{2T\lVert\nabla \bar{u}^m\rVert_{L^{\infty}([0,T]\times D)}}T\lVert\nabla\bar{u}^m\rVert_{L^{\infty}([0,T]\times D)}\leq \frac{1}{m}.$$
\end{remark}
\begin{lemma}\label{Lemma weak hp weak th}
Under the same assumptions of Theorem \ref{Theorem weak strong}
$$\lim_{n \rightarrow +\infty}\sup_{t\in [0,T]}\mathbb{E}\left[\lVert u^{n}-u\rVert^2\right]= 0. $$
\end{lemma}
\begin{proof}
Let $\{\bar{u}_0^m\}_{m\in\mathbb{N}}$ approximating $u_0$ in the sense of Theorem \ref{P.L. Lions} and $\{\bar{u}_m\}_{m\in\mathbb{N}}$ the corresponding solutions of the Euler equations, then for each $t,\ n,\ m$ we have
\begin{align*}
    \mathbb{E}\left[\lVert u^n(t)-u(t)\rVert^2\right]& \leq 2 \mathbb{E}\left[\lVert u^n(t)-\bar{u}^m(t)\rVert^2\right]+2\lVert \bar{u}^m(t)-u(t)\rVert^2\\ &\stackrel{Thm\ \ref{P.L. Lions}}{\leq}\frac{2}{m^2}+2 \mathbb{E}\left[\lVert u^n(t)-\bar{u}^m(t)\rVert^2\right].
\end{align*}
We adapt the computations of the proof of Proposition \ref{weak result strong assump} to analyze the second term, hence some explanation will be omitted. For each $m$ and $\delta>0$ fixed, let us introduce the corrector of the boundary layer $v_m$. $v_m$ satisfies previous estimates with respect to a constant dependent from $m$ and independent from $t$, namely
$$\lVert v_m(t)\rVert_{L^{\infty}(D)}\leq K_m, \ \lVert v_m(t)\rVert_{L^2(D)}\leq K_m\delta^{\frac{1}{2}}, \ \lVert\partial_t v_m(t)\rVert_{L^2(D)}\leq K_m\delta^{\frac{1}{2}},$$  $$\lVert\nabla v_m(t)\rVert_{L^{\infty}(D)}\leq K_m\delta^{-1}, 
\  \lVert\nabla v_m(t)\rVert_{L^2(D)}\leq K_m\delta^{-1/2},  \  \lVert\rho(t)\nabla v_m(t)\rVert_{L^{\infty}(D)}\leq K_m,$$ $$  \lVert\rho(t)^2 \nabla v_m(t)\rVert_{L^{\infty}(D)}\leq K_m\delta, \  \lVert\rho(t) \nabla v_m(t)\rVert_{L^2(D)}\leq K_m\delta^{\frac{1}{2}}.$$

We have at time $t$
\begin{align*}
  \mathbb{E}\left[\lVert u^n-\bar{u}^m\rVert^2\right]&= \mathbb{E}\left[\lVert u^{n}\rVert^2\right]+\lVert\bar{u}^m\rVert^2-2\mathbb{E}\left[\langle u^{n},\bar{u}^m\rangle\right]\\ & \stackrel{energy \ eq.}{\leq}  \mathbb{E}\left[\lVert u_0^{n}\rVert^2\right]+t\nu_n\sum_{k}\lVert \sigma_k\rVert^2+\lVert\bar{u}^m_0\rVert^2-2\mathbb{E}\left[\langle u^{n},\bar{u}^m-v_m\rangle\right]+2\mathbb{E}\left[\langle u^{n},v_m\rangle\right] \\ & =
    \mathbb{E}\left[\lVert u_0^{n}-u_0\rVert^2\right]+\lVert u_0\rVert^2+2\mathbb{E}\left[\langle u_0^n-u_0,u_0\rangle\right]+t\nu_n\sum_{k}\lVert \sigma_k\rVert^2+\lVert\bar{u}^m_0-u_0\rVert^2+\lVert u_0\rVert^2\\ &+2\langle\bar{u}_0^m-u_0,u_0\rangle-2\mathbb{E}\left[\langle u^{n},\bar{u}^m-v_m\rangle\right]+2\mathbb{E}\left[\langle u^{n},v_m\rangle\right]
    \\ & \leq \mathbb{E}\left[\lVert u_0^{n}-u_0\rVert ^2\right]+2\lVert u_0\rVert^2+2\mathbb{E}\left[\lVert u_0^n-u_0\rVert^2\right]^{1/2}\lVert u_0\rVert+\lVert\bar{u}_0^m-u_0\rVert^2+2\lVert\bar{u}_0^m-u_0\rVert\lVert u_0\rVert\\ &+t\nu_n\sum_{k}\lVert \sigma_k\rVert^2-2\mathbb{E}\left[\langle u^{n},\bar{u}^m-v_m\rangle\right]+2\mathbb{E}\left[\langle u^{n},v_m\rangle\right]
    \\ & \stackrel{\lVert v_m\rVert_{L^2(D)}\leq K_m\delta^{\frac{1}{2}}}{\leq} \mathbb{E}\left[\lVert u_0^{n}-u_0\rVert^2\right]+2\lVert u_0\rVert^2+K\mathbb{E}\left[\lVert u_0^n-u_0\rVert^2\right]^{1/2}+\lVert\bar{u}_0^m-u_0\rVert^2+K\lVert\bar{u}_0^m-u_0\rVert\\ &+K\nu_n-2\mathbb{E}\left[\langle u^{n},\bar{u}^m-v_m\rangle\right]+K_m\delta^{\frac{1}{2}}.
\end{align*}
To analyze the second-last term we use the weak formulation of $u^{n}$, taking $\bar{u}^m-v_m$ as test function. We take directly the expected value of the weak formulation. Exploiting the relation $$\mathbb{E}\left[\langle u_0^n,\bar{u}^m_0\rangle\right]=\lVert u_0\rVert^2+\mathbb{E}\left[\langle u_0^n-u_0,\bar{u}_0^m-u_0\rangle\right]+\mathbb{E}\left[\langle u_0^n-u_0,u_0\rangle\right]+\langle u_0,\bar{u}_0^m-u_0\rangle, $$
we get

\begin{align*} -2\mathbb{E}\left[\langle u^{n}(t),(u^m-v_m)(t)\rangle\right]+2\mathbb{E}\left[\lVert u_0\rVert^2\right] & = -2\mathbb{E}\left[\langle u_0^n-u_0,\bar{u}_0^m-u_0\rangle\right]-2\mathbb{E}\left[\langle u_0^n-u_0,u_0\rangle\right]\\ & -2\langle u_0,\bar{u}_0^m-u_0\rangle+2\mathbb{E}\left[\langle u_0^n,v_m(0)\rangle\right]\\ &  -2\mathbb{E}\left[\int_0^t\langle u^{n}(s),\partial_s(\bar{u}^m-v_m)(s)\rangle\ ds\right]\\ & +  2\nu_n\mathbb{E}\left[\int_0^t \langle (-A)^{\frac{1}{2}}u^{n}(s), (-A)^{\frac{1}{2}}(\bar{u}^m-v_m)(s)\rangle\ ds\right] \\ &- \mathbb{E}\left[\int_0^t 2b(u^{n}(s),(\bar{u}^m-v_m)(s),u^{n}(s))\ ds\right]\\ & \stackrel{\lVert v_m\rVert_{L^2(D)}\leq K_m\delta^{\frac{1}{2}}}{\leq}  2\lVert \bar{u}_0^m-u_0\rVert \mathbb{E}\left[\lVert u_0^n-u_0\rVert^2\right]^{1/2}\\ &+2\lVert u_0\rVert\mathbb{E}\left[\lVert u_0^n-u_0\rVert^2\right]^{1/2} +2\lVert u_0\rVert\lVert\bar{u}_0^m-u_0\rVert+2\mathbb{E}\left[\lVert u_0^n\rVert^2\right]^{1/2}K_m\delta^{\frac{1}{2}}\\ &  -2\mathbb{E}\left[\int_0^t\langle u^{n}(s),\partial_s(\bar{u}^m-v_m)(s)\rangle\ ds\right]\\ & +  2\nu_n\mathbb{E}\left[\int_0^t \langle (-A)^{\frac{1}{2}}u^{n}(s), (-A)^{\frac{1}{2}}(\bar{u}^m-v_m)(s)\rangle\ ds\right] \\ &- \mathbb{E}\left[\int_0^t 2b(u^{n}(s),(\bar{u}^m-v_m)(s),u^{n}(s))\ ds \right] .
\end{align*}
Moreover
\begin{align*}
    -\mathbb{E}\left[\langle u^{n}(s),\partial_s(\bar{u}^m-v_m)(s)\rangle\right] & \stackrel{energy \ eq., \ \lVert\partial_t v_m(t)\rVert_{L^2(D)}\leq K_m\delta^{\frac{1}{2}}}{=}-\mathbb{E}\left[\langle u^{n}(s),\partial_s \bar{u}^m(s)\rangle\right]+K_m\delta^{\frac{1}{2}} \\ & \stackrel{Euler\  eq}{=} \mathbb{E}\left[\langle u^{n}(s),\nabla\bar{u}^m(s)\nabla\bar{u}^m\rangle\right]+K_m\delta^{\frac{1}{2}}.
\end{align*}
Thanks to previous relations and noting that $$ b(u^{n},\bar{u}^m,u^{n})-b(u^{n},\bar{u}^{m},\bar{u}^{m})=b(u^{n}-\bar{u}^{m},\bar{u}^{m},u^{n}-\bar{u}^{m})$$ we can continue the estimate of $\mathbb{E}\left[\lVert u^{n}-\bar{u}^m\rVert^2\right]$:
\begin{align*}
     \mathbb{E}\left[\lVert u^{n}-\bar{u}^m\rVert^2\right]\leq & K_m\delta^{\frac{1}{2}}+K\mathbb{E}\left[\lVert u_0^{n}-u_0\rVert^2\right]+\lVert\bar{u}_0^m-u_0\rVert^2+K\lVert\bar{u}_0^m-u_0\rVert+K\nu_n\\ &  +K\mathbb{E}\left[\lVert u_0^n-u_0\rVert^2\right]^{1/2}+2\lVert\bar{u}_0^m-u_0\rVert\mathbb{E}\left[\lVert u_0^n-u_0\rVert^2\right]^{1/2}\\ &   -2\mathbb{E}\left[\int_0^t b(u^n-\bar{u}^m,\bar{u}^m,u^n-\bar{u}^m)\ ds\right] + 2\mathbb{E}\left[\int_0^t b(u^n,v_m,u^n)\ ds\right] \\ & +  2\nu_n\mathbb{E}\left[\int_0^t \langle(-A)^{\frac{1}{2}}u^{n}(s), (-A)^{\frac{1}{2}}(\bar{u}^m-v_m)(s)\rangle\ ds\right]\\ & \leq
     K_m\delta^{\frac{1}{2}}+K\mathbb{E}[\lVert u_0^{n}-u_0\rVert^2]+\lVert\bar{u}_0^m-u_0\rVert^2+K\lVert\bar{u}_0^m-u_0\rVert+K\nu_n\\ &  +K\mathbb{E}\left[\lVert u_0^n-u_0\rVert^2\right]^{1/2}+2\lVert\bar{u}_0^m-u_0\rVert\mathbb{E}\left[\lVert u_0^n-u_0\rVert^2\right]^{1/2}\\ & + 2\mathbb{E}\left[\int_0^t b(u^n,v_m,u^n)\ ds \right]\\ & +  2\nu_n\mathbb{E}\left[\int_0^t \langle(-A)^{\frac{1}{2}}u^{n}(s), (-A)^{\frac{1}{2}}(\bar{u}^m-v_m)(s)\rangle\ ds\right]\\ & +2\lVert \nabla\bar{u}^m\rVert_{L^{\infty}([0,T]\times D)}\mathbb{E}\left[\int_0^t\lVert u^n-\bar{u}^m\rVert^2\ ds \right].
\end{align*}
Arguing as in the proof of Proposition \ref{weak result strong assump} we have
$$\mathbb{E}\left[\int_0^t b(u^n,v_m,u^n)\ ds\right]\stackrel{\lVert\rho^2(t)\nabla v_m(t)\rVert_{L^{\infty}}\leq K_m\delta}{\leq} K_m\delta \mathbb{E}\left[\int_0^T\lVert \nabla u^n\rVert_{L^2(\Gamma_{\delta})}^2\ ds\right]$$
\begin{align*}
& \nu_n\mathbb{E}\left[\int_0^t \langle(-A)^{\frac{1}{2}}u^{n}(s), (-A)^{\frac{1}{2}}(\bar{u}^m-v_m)(s)\rangle\ ds \right]\\ & \stackrel{\lVert\nabla v^m(t)\rVert_{L^2(D)}\leq K_m\delta^{-1/2}}{\leq} \lVert\nabla \bar{u}^m\rVert_{L^{\infty}(0,T;L^2(D))}\nu_n\mathbb{E}\left[\int_0^T\lVert\nabla u^n\rVert_{L^2(D)}\ ds\right]+\nu_nK_m\delta^{-1/2}\mathbb{E}\left[\int_0^T \lVert\nabla u^n\rVert_{L^2(\Gamma_{\delta})}\ ds\right].
\end{align*}
Thanks to this relations we can continue the estimate of $\mathbb{E}\left[\lVert u^{n}-\bar{u}^m\rVert^2\right]$:
\begin{align*}
      \mathbb{E}\left[\lVert u^{n}-\bar{u}^m\rVert^2\right]  & \leq
     K_m\delta^{\frac{1}{2}}+K\mathbb{E}[\lVert u_0^{n}-u_0\rVert^2]+\lVert\bar{u}_0^m-u_0\rVert^2+K\lVert\bar{u}_0^m-u_0\rVert+K\nu_n\\ &  +K\mathbb{E}\left[\lVert u_0^n-u_0\rVert^2\right]^{1/2}+2\lVert\bar{u}_0^m-u_0\rVert\mathbb{E}\left[\lVert u_0^n-u_0\rVert^2\right]^{1/2} + K_m\delta \mathbb{E}\left[\int_0^T\lVert\nabla u^n\rVert_{L^2(\Gamma_{\delta})}^2\ ds\right] \\ & + \lVert\nabla \bar{u}^m\rVert_{L^{\infty}(0,T;L^2(D))}\nu_n\mathbb{E}\left[\int_0^T\lVert\nabla u^n\rVert_{L^2(D)}\ ds\right]+\nu_nK_m\delta^{-1/2}\mathbb{E}\left[\int_0^T \lVert\nabla u^n\rVert_{L^2(\Gamma_{\delta})}\ ds\right]  \\ & +2\lVert\nabla\bar{u}^m\rVert_{L^{\infty}([0,T]\times D)}\mathbb{E}\left[\int_0^t\lVert u^n-\bar{u}^m\rVert^2\ ds \right].
\end{align*}
Taking $\delta=c\nu_n$, by Gronwall's inequality and Holder's inequality we have
\begin{align*}
      sup_{t\in[0,T]}\mathbb{E}\left[\lVert u^{n}-\bar{u}^m\rVert^2 \right] & \leq (
     K_m\nu_n^{1/2}+K\mathbb{E}[\lVert u_0^{n}-u_0\rVert^2]+\lVert\bar{u}_0^m-u_0\rVert^2+K\lVert\bar{u}_0^m-u_0\rVert+K\nu_n\\ &  +K\mathbb{E}\left[\lVert u_0^n-u_0\rVert^2\right]^{1/2}+ K_m\nu_n \mathbb{E}\left[\int_0^T\lVert\nabla u^n\rVert_{L^2(\Gamma_{\delta})}^2\ ds\right] \\ & + \lVert\nabla \bar{u}^m\rVert_{L^{\infty}(0,T;L^2(D))}\sqrt{\nu_n}\mathbb{E}\left[\int_0^T\nu_n\lVert\nabla u^n\rVert_{L^2(D)}^2\ ds\right]^{1/2}\\ &+K_m\mathbb{E}\left[\nu_n \int_0^T \lVert\nabla u^n\rVert^2_{L^2(\Gamma_{\delta})}\ ds\right]^{\frac{1}{2}} )  e^{2T\lVert\nabla \bar{u}^m\rVert_{L^{\infty}([0,T]\times D)}}.
\end{align*}

Taking the limsup with respect to $n$ of this expression for $m$ fixed we have
\begin{align}\label{estimate final thm}
    \limsup_{n\rightarrow+\infty} \sup_{t\in[0,T]}\mathbb{E}\left[\lVert u^{n}-\bar{u}^m\rVert^2\right]\leq K\lVert\bar{u}^m_0-u_0\rVert e^{2T\lVert\nabla \bar{u}^m\rVert_{L^{\infty}([0,T\times D])}}\stackrel{Remark \ \ref{Remark dis}}{\leq} \frac{K}{m}e^{-T\lVert\nabla \bar{u}^m\rVert_{L^{\infty}([0,T]\times D)}}.
\end{align}
Coming back to \begin{align*}
    \mathbb{E}\left[\lVert u^n(t)-u(t)\rVert^2\right]& \leq \frac{2}{m^2}+2 \mathbb{E}\left[\lVert u^n(t)-\bar{u}^m(t)\rVert^2\right].
\end{align*}
If we fix $\epsilon>0$ and $\bar{m}$ such that $2\frac{1+K\bar{m}}{\bar{m}^2}<\epsilon$, then taking the limsup with respect to $n$ of previous expression for $m=\bar{m}$ we have 
\begin{align*}
    \limsup_{n\rightarrow+\infty} \mathbb{E}\left[\lVert u^n(t)-u(t)\rVert^2\right]\leq \epsilon.
\end{align*}
We have the thesis from the arbitrariness of $\epsilon$.
\end{proof}\\

\begin{proof}[Proof of Theorem \ref{Theorem weak strong}]
Let $\{\bar{u}_0^m\}_{m\in\mathbb{N}}$ approximating $u_0$ in the sense of Theorem \ref{P.L. Lions} and $\{\bar{u}_m\}_{m\in\mathbb{N}}$ the corresponding solutions of the Euler equations, then for each $t,\ n,\ m$ we have\begin{align*}
    \lVert u^n(t)-u(t)\rVert^2& \leq 2 \lVert u^n(t)-\bar{u}^m(t)\rVert^2+2\lVert \bar{u}^m(t)-u(t)\rVert^2\\ &\stackrel{Thm\ \ref{P.L. Lions}}{\leq}\frac{2}{m^2}+2 \lVert u^n(t)-\bar{u}^m(t)\rVert^2.
\end{align*}
We adapt the computations of the proof of Corollary \ref{Strong result strong assum} and Lemma \ref{Lemma weak hp weak th}  to analyze the last term, hence some explanation will be omitted. For each $m$ and $\delta>0$ fixed, let us introduce the corrector of the boundary layer $v_m$, it satisfies previous estimates. We have at time $t$
\begin{align*}
    \lVert u^{n}-\bar{u}^m\rVert^2 & =\lVert u^{n}\rVert^2+\lVert\bar{u}^m\rVert^2-2\langle u^{n},\bar{u}^m\rangle\\
    & \stackrel{It\hat{o}}{\leq}\lVert u_0^{n}\rVert^2+t\nu_n\sum_{k=1}^N\lVert \sigma_k \rVert^2+2\sum_{k=1}^N\nu_n^{\frac{1}{2}}\int_0^t\langle u^{n}(s),\sigma_k\rangle dW^k_s\\ &+\lVert\bar{u}^m_0\rVert^2-2\langle u^{n},\bar{u}^m-v_m\rangle-2\langle u^n,v_m\rangle. 
\end{align*}
Let us rewrite $\langle u^{n},\bar{u}^m-v_m\rangle$ thanks to the weak formulation of $u^{n}$
\begin{align*}
    -2\langle u^{n},\bar{u}^m-v_m\rangle& =-2\langle u^{n}_0,(\bar{u}^m-v_m)(0)\rangle-2\int_0^t\langle u^{n}(s),\partial_s(\bar{u}^m-v_m)(s)\rangle\ ds\\ &+ 2\nu_n\int_0^t\langle(-A)^{\frac{1}{2}}u^{n}(s),(-A)^{\frac{1}{2}}(\bar{u}^m-v_m)(s)\rangle\ ds-2\int_0^t b(u^{n},\bar{u}^m-v_m,u^n)(s)\ ds\\
    & -2\sqrt{\nu_n}\sum_{k=1}^N\langle\sigma_k,(\bar{u}^m-v_m)\rangle W^k_t+2\sqrt{\nu_n}\sum_{k=1}^N\int_0^t\langle\sigma_k,(\bar{u}^m-v_m)(s)\rangle W^k_s\ ds.
\end{align*}
Moreover \begin{align*}
-2\langle u^{n}(s),\partial_s(\bar{u}^m-v_m)(s)\rangle=2\langle u^{n}(s),\partial_s v^m(s)\rangle+2\langle u^{n}(s),\nabla \bar{u}^m\cdot \bar{u}^m(s)\rangle. 
\end{align*}
Thanks to previous relations and noting that $$b(u^{n},\bar{u}^m,u^{n})-b(u^{n},\bar{u}^m,\bar{u}^m)=b(u^{n}-\bar{u}^m,\bar{u}^m,u^{n}-\bar{u}^m)$$ we have
\begin{align*}
    \lVert u^{n}-\bar{u}^m\rVert^2& =(\lVert u_0^{n}\rVert^2+\lVert\bar{u}_0^m\rVert^2-2\langle u^{n}_0,(\bar{u}^m-v_m)(0)\rangle)+(t\nu_n\sum_{k=1}^N\lVert \sigma_k \rVert^2+2\sqrt{\nu_n}\sum_{k=1}^N\int_0^t\langle u^{n}(s),\sigma_k\rangle dW^k_s\\ &-2\sqrt{\nu_n}\sum_{k=1}^N\langle \sigma_k,(\bar{u}^m-v_m)(t)\rangle W^k_t+2\sqrt{\nu_n}\sum_{k=1}^N\int_0^t \langle \sigma_k,(\bar{u}^m-v_m)(s)\rangle W^k_s \ ds)\\ &+(2\int_0^t b(u^{n},v_m,u^{n})(s) \ ds-2\int_0^t b(u^{n}-\bar{u}^m,\bar{u}^m,u^{n}-\bar{u}^m)(s)\ ds)+\\ & (-2\langle u^{n},v_m\rangle+2\nu_n \int_0^t\langle (-A)^{\frac{1}{2}}u^{n}(s),(-A)^{\frac{1}{2}}(\bar{u}^m-v_m)(s)\rangle\ ds+2\int_0^t \langle u^{n},\partial_s v_m\rangle\ ds)\\ &= I_1(t)+I_2(t)+I_3(t)+I_4(t).
\end{align*}
Thus $$\mathbb{E}\left[sup_{t\in [0,T]}\lVert u^{n}-\bar{u}^m\rVert^2\right]\leq \mathbb{E}\left[sup_{t\in [0,T]}I_1\right]+\mathbb{E}\left[sup_{t\in [0,T]}I_2\right]+\mathbb{E}\left[sup_{t\in [0,T]}I_3\right]+\mathbb{E}\left[sup_{t\in [0,T]}I_4\right]$$
\begin{align*}
    \mathbb{E}\left[\sup_{t\in [0,T]} I_1\right]&= \mathbb{E}\left[\lVert u_0^{n}\rVert^2+\lVert\bar{u}^m_0\rVert^2-2\langle u^{n}_0,(\bar{u}^m-v_m)(0)\rangle\right]\\ & = -\mathbb{E}\left[\lVert u_0^n\rVert^2\right]+\lVert\bar{u}_0^m\rVert^2-2\mathbb{E}\left[\langle u_0^n,\bar{u}_0^m-u^n_0\rangle\right] +2\mathbb{E}\left[\langle u_0^n,v_m(0)\rangle\right]    \\  & \leq -\mathbb{E}\left[\lVert u_0-u_0^n\rVert^2\right]-\lVert u_0\rVert^2-2\mathbb{E}\left[\langle u_0^n-u_0,u_0\rangle\right]+\lVert u_0\rVert^2+\lVert u_0-\bar{u}_0^m\rVert^2\\ & +2\langle u_0,\bar{u}^m_0-u_0\rangle+2\mathbb{E}\left[\lVert u_0^n\rVert\lVert\bar{u}_0^m-u^n_0\rVert\right] +2\mathbb{E}\left[\lVert u_0^n\rVert\right]\lVert v_m(0)\rVert \\ \\  & \stackrel{\lVert v_m(t)\rVert_{L^2(D)}\leq K_m\delta^{\frac{1}{2}}}{\leq}\mathbb{E}\left[\lVert u_0^n-u_0\rVert^2\right]+K\mathbb{E}\left[\lVert u_0^n-u_0\rVert^2\right]^{1/2}+ \lVert\bar{u}^m_0-u_0\rVert^2\\ & +K\lVert\bar{u}^m_0-u_0\rVert +K_m\delta^{\frac{1}{2}}\mathbb{E}\left[\lVert u^n_0\rVert^2\right]^{1/2} 
   .
\end{align*}
The analysis of the others is similar to what we have done before:
\begin{align*}
    \mathbb{E}\left[sup_{t\in[0,T]} I_3\right] &\leq 2 \mathbb{E}\left[sup_{t\in[0,T]}\int_0^t b(u^{n},v^m,u^{n})(s)\ ds\right]+2 \mathbb{E}\left[sup_{t\in[0,T]}\int_0^t b(u^{n}-\bar{u}^m,\bar{u}^m,u^{n}-\bar{u}^m)(s)\ ds\right]\\
    &\leq 2 \mathbb{E}\left[\int_0^T|b(u^{n},v^m,u^{n})(s)|\ ds\right]+2 \mathbb{E}\left[\int_0^T|b(u^{n}-\bar{u}^m,\bar{u}^m,u^{n}-\bar{u}^m)(s)|\ ds\right]\\
    & \leq2\mathbb{E}\left[\int_0^T \lVert\rho^2\nabla v_m(s)\rVert_{L^{\infty}(D)}\lVert u^{n}\rho^{-1}(s)\rVert_{L^2(D)}^2\right]+2\mathbb{E}\left[\int_0^T\lVert u^{n}-\bar{u}^m\rVert^2(s)\lVert \nabla\bar{u}^m(s)\rVert_{L^{\infty}(D)}\ ds\right]\\ &\leq 2K_m\delta \mathbb{E}\left[\int_0^T\lVert\nabla u^{n}(s)\rVert_{L^2(\Gamma_{\delta})}^2\ ds\right]+2\lVert\nabla \bar{u}^m\rVert_{L^{\infty}([0,T]\times D)}\mathbb{E}\left [\int_0^T\lVert u^{n}-\bar{u}^m(s)\rVert^2\ ds \right]
\end{align*}

Let us analyze all the elements of $I_2$ exploiting previous energy equalities and properties of Brownian motion:
\begin{align*}
    \mathbb{E}\left[\sup_{t\in[0,T]} T\nu_n \sum_{k=1}^N \lVert\sigma_k\rVert^2\right]<K\nu_n
\end{align*}
\begin{align*}
    \mathbb{E}\left[sup_{t\in[0,T]} 2\sum_{k=1}^N \sqrt{\nu_n}\int_0^t \langle u^{n},\sigma_k\rangle dW^k_s\right]\stackrel{Doob's\  ineq}{\leq}K\sqrt{\nu_n}
\end{align*}
\begin{align*}
    \mathbb{E}\left[sup_{t\in[0,T]} 2\sqrt{\nu_n}\sum_{k=1}^N\langle \sigma_k,\bar{u}^m-v_m(t)\rangle W^k_t\right] & \leq K\sqrt{\nu_n}\lVert\bar{u}^m-v_m\rVert_{L^{\infty}(0,T;L^2(D))}\mathbb{E}\left[sup_{t\in[0,T]}|W^k_t|\right]\\ &\leq K\sqrt{\nu_n}\lVert\bar{u}^m-v_m\rVert_{L^{\infty}(0,T;L^2(D))}
\end{align*}
\begin{align*}
    \mathbb{E}\left[\sup_{t\in[0,T]} 2\sqrt{\nu_n}\sum_{k=1}^N \int_0^t \langle \sigma_k,u-v(s)\rangle W^k_s\ ds\right]\leq K\sqrt{\nu_n}\lVert\bar{u}^m-v_m\rVert_{L^{\infty}(0,T;L^2(D))}.
\end{align*}
It remains only to analyze $I_4$. Some of the estimates below use tricks already presented, hence some details have been omitted.
\begin{align*}
    2\nu_n \mathbb{E}\left[\sup_{t\in[0,T]}\int_0^t \langle (-A)^{\frac{1}{2}}u^{n},(-A)^{\frac{1}{2}}(\bar{u}^m-v_m)\rangle\ ds\right]& \leq 2\nu_n \mathbb{E}\left[\int_0^T \lVert\nabla \bar{u}^m \rVert_{L^{\infty}(0,T;L^2(D))}\lVert\nabla u^{n}(s)\rVert\ ds \right]\\ & +2\nu_n\mathbb{E}\left[\int_0^T K_m\delta^{-\frac{1}{2}}\lVert\nabla u^{n}(s)\rVert_{L^2(\Gamma_{\delta})}\ ds\right]
\end{align*}
\begin{align*}
    \mathbb{E}\left[\sup_{t\in[0,T]} \langle u^{n},v_m\rangle\right]\leq \mathbb{E}\left[\sup_{t\in[0,T]} \lVert u^{n}\rVert\sup_{t\in[0,T]}\lVert v_m\rVert \right]\leq K_m \delta^{\frac{1}{2}}\mathbb{E}\left[\sup_{t\in[0,T]} \lVert u^{n}\rVert^2\right]^{\frac{1}{2}}\leq K_m\delta^{\frac{1}{2}}
\end{align*}
\begin{align*}
    \mathbb{E}\left[sup_{t\in[0,T]} \int_0^t \langle u^{n},\partial_s v_m\rangle\ ds\right]\leq \lVert\partial_s v_m\rVert_{L^{\infty}(0,T;L^2(D))}\mathbb{E}\left[\int_0^T\lVert u^{n}(s)\rVert\ ds \right]\leq K_m \delta^{\frac{1}{2}}.
\end{align*}
In conclusion, if we take $\delta=c\nu$, then by Holder's inequality \begin{align*}
    \mathbb{E}[\sup_{t\in [0,T]}\lVert u^{n}-\bar{u}^m\rVert^2] & \leq \mathbb{E}\left[\lVert u_0^n-u_0\rVert^2\right]+K\mathbb{E}\left[\lVert u_0^n-u_0\rVert^2\right]^{1/2}+\lVert\bar{u}_0^m-u_0\rVert^2+K\lVert\bar{u}_0^m-u_0\rVert\\ & +K_m\sqrt{\nu_n}+K\nu_n+K\sqrt{\nu_n}+K\sqrt{\nu_n}\lVert\bar{u}_m-v_m\rVert_{L^{\infty}(0,T;L^2(D))} \\ & +2K_m\nu_n\mathbb{E}\left[\int_0^T\lVert\nabla u^n(s)\rVert_{L^2(\Gamma_{c\nu_n})}^2\ ds\right]+2\lVert\nabla\bar{u}^m\rVert_{L^{\infty}([0,T]\times D)}\mathbb{E}\left[\int_0^T\lVert u^n-\bar{u}^m(s)\rVert^2 \ ds\right]\\ & +\lVert\nabla\bar{u}^m\rVert_{L^{\infty}(0,T;L^2(D))}\sqrt{\nu_n}\mathbb{E}\left[\nu_n \int_0^T \lVert\nabla u^n(s)\rVert_{L^2(D)}^2\ ds\right]^{1/2}\\ &+2K_m\mathbb{E}\left[ \int_0^T \nu_n \lVert\nabla u^n\rVert_{L^2(\Gamma_{\nu_n})}^2\ ds\right]^{1/2}+K_m\nu_n.
\end{align*} 
Taking the limsup with respect to $n$ of this expression for $m$ fixed we have \begin{align*}
    \limsup_{n\rightarrow+\infty}\mathbb{E}\left[\sup_{t\in [0,T]}\lVert u^{n}-\bar{u}^m\rVert^2\right] & \leq \lVert\bar{u}_0^m-u_0\rVert^2+K\lVert\bar{u}_0^m-u_0\rVert\\ & +2T\lVert\nabla\bar{u}^m\rVert_{L^{\infty}([0,T]\times D)}\limsup_{n\rightarrow +\infty}\left(\sup_{t\in[0,T]}\mathbb{E}\left[\lVert u^n-\bar{u}_m\rVert^2\right]\right)\\ &
    \stackrel{eq.\ (\ref{estimate final thm})}{\leq} \lVert\bar{u}_0^m-u_0\rVert^2+K\lVert\bar{u}_0^m-u_0\rVert\\ & +2T\lVert \nabla\bar{u}^m\rVert_{L^{\infty}([0,T]\times D)}\frac{K}{m}e^{-T\lVert\nabla \bar{u}^m\rVert_{L^{\infty}([0,T]\times D)}}\\&  \stackrel{Remark\ \ref{Remark dis}}{\leq} \lVert\bar{u}_0^m-u_0\rVert^2+K\lVert\bar{u}_0^m-u_0\rVert+\frac{K}{m}.
\end{align*}
Coming back to \begin{align*}
    \lVert u^n(t)-u(t)\rVert^2\leq  \frac{2}{m^2}+2 \lVert u^n(t)-\bar{u}^m(t)\rVert^2.
\end{align*}
If we fix $\epsilon>0$ and $\bar{m}$ such that $$2\lVert\bar{u}_0^{\bar{m}}-u_0\rVert^2+2K\lVert\bar{u}_0^{\bar{m}}-u_0\rVert+2\frac{K\bar{m}+1}{\bar{m}^2}<\epsilon,$$
then taking the expected value of the supremum in time of the previous expression for $m=\bar{m}$ we have 
\begin{align*}
     \mathbb{E}\left[\sup_{t\in[0,T]}\lVert u^n(t)-u(t)\rVert^2\right]\leq  \frac{2}{\bar{m}^2}+2\mathbb{E}\left[\sup_{t\in[0,T]} \lVert u^n(t)-\bar{u}^{\bar{m}}(t)\rVert^2\right].
\end{align*}
Taking the limsup with respect to $n$ of the last inequality we have
$$\limsup_{n\rightarrow+\infty}\mathbb{E}\left[\sup_{t\in[0,T]}\lVert u^n(t)-u(t)\rVert^2\right]<\epsilon. $$  We have the thesis from the arbitrariness of $\epsilon$.
\end{proof}
\section{A Deterministic Remark} \label{Sec Rem}
As anticipated in Remark \ref{remark novelty}, in this section we prove an inviscid limit result in the deterministic framework, analogous to Theorem \ref{Theorem weak strong} for a particular class of external forces. This result extends the setting considered by Kato in \cite{kato1984remarks} and it is the object of Theorem \ref{deterministc thm}.
\begin{lemma}\label{P.L. Lions style} 
If $u$ is a weak solution of the Euler equations with initial condition $u_0\in H$ and external force $f\in L^2(0,T;H)$ and $\Bar{u}$ is the unique weak solution of the Euler equations with initial condition $\Bar{u}_0\in H\cap C^{1,\epsilon}(\bar{D})$ and external force $\Bar{f}\in L^2(0,T;H)\cap C^{1,\epsilon}([0,T]\times\bar{D}) $, then \begin{align*}
    \lVert (u-\bar{u})(t)\rVert^2 \leq& e^{2t\lVert\nabla \bar{ u}\rVert_{L^{\infty}([0,T]\times D)}}(\lVert u_0- \bar{u}_0\rVert^2\\ &+2\sqrt{T}\lVert f-\bar{f}\rVert_{L^2(0 ,T;H)}  \left(\sqrt{2\lVert u_0\rVert^2+4T\lVert f\rVert_{L^2(0,T;H)}^2}+\sqrt{2\lVert \bar{u}_0\rVert^2+4T\lVert \bar{f}\rVert_{L^2(0,T;H)}^2}\right)).\end{align*}
For each $K\geq 1$, calling
    \begin{align*}
     O_n^K = &\{u_0\in H,\ f\in L^2(0,T;H): \ \exists \bar{u_0}\in H\cap C^{1,\epsilon}(\bar{D}),\ \Bar{f}\in L^2(0,T;H)\cap C^{1,\epsilon}([0,T]\times\bar{D}),\\ &  \lVert u_0-\bar{u}_0\rVert <\frac{1}{n}e^{-KT\lVert\nabla \bar{ u}\rVert_{L^{\infty}([0,T]\times D)}},\ \lVert f-\bar{f}\rVert <\frac{1}{n^2}e^{-2KT\lVert\nabla \bar{ u}\rVert_{L^{\infty}([0,T]\times D)}}\}
    \end{align*}
where $\bar{u}$ is the solution of the Euler equations with initial condition $\Bar{u}_0$ and external force $\bar{f}$, then for each $(u_0,f)\in \bigcap_{n\geq 1}O_n^K=:\Tilde{O^K} $ there exists a unique $u\in C([0,T],H)$ weak solution of the Euler equations with initial condition $u_0$ and external force ${f}$. Moreover the energy equality $$\lVert u(t)\rVert^2=\lVert u_0\rVert^2+2\int_0^t\langle f, u\rangle\ ds $$ holds.
\end{lemma}
\begin{proof}
\begin{itemize}
    \item [Estimate:] For what concern the solution with smooth initial condition and external force, thanks to Theorem \ref{Kato classiche}, $\bar{u}$ is a classical solution and the energy equality holds, namely \begin{align*}
    \lVert \bar{u}(t)\rVert^2& = \lVert \bar{u}(0)\rVert^2+2\int_0^t\langle \bar{f},\bar{u}\rangle\ ds\leq \lVert \bar{u}(0)\rVert^2+2\int_0^t\lVert \bar{f}\rVert \lVert \bar{u} \rVert\ ds\\ & \stackrel{Young\ ineq. }{\leq}  \lVert \bar{u}_0\rVert^2+2T\int_0^T\lVert \bar{f}\rVert^2\ ds+\int_0^T\frac{\lVert\bar{u}\rVert^2}{2T}\ ds\\ & \leq \lVert \bar{u}_0\rVert^2+2T\lVert \bar{f}\rVert_{L^2(0,T;H)}^2+\frac{\lVert\bar{u}\rVert_{L^{\infty}(0,T;H)}^2}{2}.
\end{align*} 
Thus \begin{align}\label{A priori estim}
    \sup_{t\in [0,T]}\lVert \bar{u}\rVert^2\leq 2\lVert \bar{u}_0\rVert^2+4T\lVert \bar{f}\rVert_{L^2(0,T;H)}^2.
\end{align}
The same estimate holds for any weak solution of the Euler equations (if it exists) with initial condition $u_0$ and external force $f$ thanks to energy inequality and the same computations. Let us consider the weak formulation satisfied by $u$ using as test function $\bar{u}$. Then at time $t$
\begin{align*}
    \langle u, \bar{u}\rangle & =\langle u_0,\bar{u}_0\rangle+\int_0^t\langle u,\partial_s \bar{u}\rangle \ ds+\int_0^t b(u,\bar{u},u)\ ds+\int_0^t \langle f,\bar{u}\rangle \ ds\\ &= \langle u_0,\bar{u}_0\rangle+\int_0^t b(u-\bar{u},\bar{u},u-\bar{u})\ ds+\int_0^t \langle f,\bar{u}\rangle \ ds+\int_0^t \langle \bar{f},u\rangle \ ds\\ & \leq \langle u_0,\bar{u}_0\rangle+\lVert \nabla\bar{u}\rVert_{L^{\infty}(0,T;L^{\infty}(D))}\int_0^t \lVert u-\bar{u}\rVert^2 \ ds+\int_0^t \langle f,\bar{u}\rangle \ ds+\int_0^t \langle \bar{f},u\rangle \ ds.
\end{align*}
Thus \begin{align*}
    \lVert u-\bar{u} \rVert^2 & = \lVert u\rVert^2+\lVert \bar{u}\rVert^2-2\langle u,\bar{u}\rangle\\  & \stackrel{energy\  ineq.}{\leq} \lVert\bar{u}_0\rVert^2+2\int_0^t\langle \bar{f},\bar{u}\rangle\ ds+\lVert{u}_0\rVert^2+2\int_0^t\langle {f},{u}\rangle\ ds-2\langle u,\bar{u}\rangle\\  &
    \stackrel{weak\ form.}{\leq} \lVert u_0-\bar{u}_0\rVert^2+2\lVert \nabla\bar{u}\rVert_{L^{\infty}(0,T;L^{\infty}(D))}\int_0^t \lVert u-\bar{u}\rVert^2 \ ds+2\int_0^t\lVert u-\bar{u}\rVert\lVert f-\bar{f}\rVert\ ds\\ & \leq   \lVert u_0-\bar{u}_0\rVert^2+2\lVert \nabla\bar{u}\rVert_{L^{\infty}(0,T;L^{\infty}(D))}\int_0^t \lVert u-\bar{u}\rVert^2 \ ds\\ &+2\sqrt{T}\lVert f-\bar{f}\rVert_{L^2(0,T;H)}(\lVert u\rVert_{L^{\infty}(0,T;H)}+\lVert \bar{u}\rVert_{L^{\infty}(0,T;H)}) \\ & \stackrel{(\ref{A priori estim})}{\leq} \lVert u_0-\bar{u}_0\rVert^2+2\lVert \nabla\bar{u}\rVert_{L^{\infty}(0,T;L^{\infty}(D))}\int_0^t \lVert u-\bar{u}\rVert^2 \ ds\\ &+2\sqrt{T}\lVert f-\bar{f}\rVert_{L^2(0,T;H)}\left(\sqrt{2\lVert \bar{u}_0\rVert^2+4T\lVert \bar{f}\rVert_{L^2(0,T;H)}^2}+\sqrt{2\lVert {u}_0\rVert^2+4T\lVert {f}\rVert_{L^2(0,T;H)}^2}\right).
\end{align*}
Thus \begin{align*}
    \lVert (u-\bar{u})(t)\rVert^2 & \leq e^{2t\lVert\nabla \bar{ u}\rVert_{L^{\infty}([0,T]\times D)}}(\lVert u_0- \bar{u}_0\rVert^2\\ &+2\sqrt{T}\lVert f-\bar{f}\rVert_{L^2(0 ,T;H)}  \left(\sqrt{2\lVert u_0\rVert^2+4T\lVert f\rVert_{L^2(0,T;H)}^2}+\sqrt{2\lVert \bar{u}_0\rVert^2+4T\lVert \bar{f}\rVert_{L^2(0,T;H)}^2}\right)).\end{align*}
\item[Existence:]
Let $(u_0,\ f)\in \Tilde{O^K}$ and $\{(\bar{u}_0^n,\bar{f}^n)\}_{n\in\mathbb{N}}$ a sequence which approximates $(u,f)$ in the sense of the theorem, namely $\bar{u}^n_0\in H\cap C^{1,\epsilon}(\bar{D}),\ \bar{f}^n\in L^2(0,T;H)\cap C^{1,\epsilon}([0,T]\times\bar{D}) $ and $$\lVert u_0-\bar{u}^n_0\rVert <\frac{1}{n}e^{-KT\lVert\nabla \bar{ u}^n\rVert_{L^{\infty}([0,T]\times D)}}$$ 
$$ \lVert f-\bar{f}^n\rVert_{L^2(0,T;H)} <\frac{1}{n^2}e^{-2KT\lVert\nabla \bar{ u}^n\rVert_{L^{\infty}([0,T]\times D)}},$$
where $\bar{u}^n$ is the solution of the Euler equations with initial condition $\Bar{u}_0^n$ and external force $\bar{f}^n$. We will prove that $\{\bar{u}^n\}$ is a Cauchy sequence in $C(0,T;H)$ and the solution of the Euler equations with initial condition $u_0$ and external force $f$ is unique. \\ Preliminarily, note that if $\lVert a-b \rVert ^2\leq \alpha,$ then $\lVert a\rVert^2\leq 4\lVert b\rVert^2+\frac{4}{3}\alpha$.\\ 
For what concern uniqueness, if $u^1$ and $u^2$ are two solutions of the Euler equations with initial condition $u_0$ and external force $f$, then at time $t$\begin{align*}
    \lVert u^1-u^2\rVert^2 & \leq  2\lVert u^1-\bar{u}^n\rVert^2+2\lVert u^2-\bar{u}^n\rVert^2\\& \leq  4  e^{2t\lVert\nabla \bar{ u}^n\rVert_{L^{\infty}([0,T]\times D)}}(\lVert u_0- \bar{u}_0^n\rVert^2\\ &+2\sqrt{T}\lVert f-\bar{f}^n\rVert_{L^2(0,T;H)}\left(\sqrt{2\lVert u_0\rVert^2+4T\lVert f\rVert_{L^2(0,T;H)}^2}+\sqrt{2\lVert \bar{u}_0^n\rVert^2+4T\lVert \bar{f}^n\rVert_{L^2(0,T;H)}^2}\right))\\ & \leq \frac{4}{n^2} \left(1+2\sqrt{T}\left(\sqrt{2\lVert u_0\rVert^2+4T\lVert f\rVert_{L^2(0,T;H)}^2}+\sqrt{8\lVert u_0\rVert^2+16T\lVert f\rVert_{L^2(0,T;H)}^2+\frac{8+16T}{3}}\right)\right).
\end{align*}

From the last inequality the uniqueness of the solution is evident.
Lastly let us consider $\lVert \bar{u}^n-\bar{u}^m\rVert^2$ for $n\geq m$. We have
\begin{align*}
\lVert \bar{u}^n-\bar{u}^m\rVert^2 & \leq    e^{2T\lVert\nabla \bar{ u}^m\rVert_{L^{\infty}([0,T]\times D)}}(\lVert \bar{u}^n_0- \bar{u}^m_0\rVert^2\\ &+2\sqrt{T}\lVert \bar{f}^n-\bar{f}^m\rVert_{L^2(0,T;H)}\left(\sqrt{2\lVert \bar{u}^m_0\rVert^2+4T\lVert \bar{f}^m\rVert_{L^2(0,T;H)}^2}+\sqrt{2\lVert \bar{u}^n_0\rVert^2+4T\lVert \bar{f}^n\rVert_{L^2(0,T;H)}^2}\right))\\  & \leq  e^{2T\lVert\nabla \bar{ u}^m\rVert_{L^{\infty}([0,T]\times D)}}(2\lVert \bar{u}^n_0-u_0\rVert^2 +2\lVert \bar{u}^m_0-u_0\rVert^2\\ &+2\sqrt{T}\lVert \bar{f}^n-f\rVert_{L^2(0,T;H)}\left(\sqrt{2\lVert \bar{u}^m_0\rVert^2+4T\lVert \bar{f}^m\rVert_{L^2(0,T;H)}^2}+\sqrt{2\lVert \bar{u}^n_0\rVert^2+4T\lVert \bar{f}^n\rVert_{L^2(0,T;H)}^2}\right)\\ & +2\sqrt{T}\lVert f-\bar{f}^m\rVert_{L^2(0,T;H)}\left(\sqrt{2\lVert \bar{u}^m_0\rVert^2+4T\lVert \bar{f}^m\rVert_{L^2(0,T;H)}^2}+\sqrt{2\lVert \bar{u}^n_0\rVert^2+4T\lVert \bar{f}^n\rVert_{L^2(0,T;H)}^2}\right))\\ & \leq  {C(T,\lVert u\rVert, \lVert f\rVert_{L^2(0,T;H)})}\left(\frac{1}{n^2}+\frac{1}{m^2}\right).
\end{align*}
The last inequality implies existence.
\item[Energy:] Let $(u_0,\ f)\in \Tilde{O^K}$ and $\{(\bar{u}_0^n,\bar{f}^n)\}_{n\in\mathbb{N}}$ a sequence which approximates $(u,f)$ in the sense of the theorem like in the previous step. Then for each $n\in \mathbb{N}$ \begin{align*}
    \lVert \bar{u}^n(t)\rVert^2=\lVert \bar{u}^n_0\rVert^2+\int_0^t\langle \bar{u}^n(s),\bar{f}^n(s)\rangle\ ds.
\end{align*}
Exploiting the fact that $\bar{u}^n\stackrel{C([0,T];H)}{\longrightarrow} u,\ \bar{f}^n\stackrel{L^2(0,T;H)}{\longrightarrow} f$ we get easily the thesis.

\end{itemize}
\end{proof}\\
Calling $\Tilde{O}:=\Tilde{O^1}$, for each $(u_0,f)\in \Tilde{O}$ we will say that $\{(\bar{u}^m_0,\bar{f}^m)\}_{m\in \mathbb{N}}$ approximates $(u_0,f)$ in the sense of Theorem \ref{P.L. Lions style} if $\bar{u}^m_0\in H\cap C^{1,\epsilon}(\bar{D}),\ \bar{f}^m\in L^2(0,T;H)\cap C^{1,\epsilon}([0,T]\times\bar{D}) $ and $$\lVert u_0-\bar{u}^m_0\rVert <\frac{1}{m}e^{-2T\lVert\nabla \bar{ u}^m\rVert_{L^{\infty}([0,T]\times D)}}$$ 
$$ \lVert f-\bar{f}^m\rVert_{L^2(0,T;H)} <\frac{1}{m^2}e^{-4T\lVert\nabla \bar{ u}^m\rVert_{L^{\infty}([0,T]\times D)}}$$
where $\bar{u}^m$ is the solution of the Euler equations with initial condition $\Bar{u}_0^m$ and external force $\bar{f}^m$.
\begin{remark}\label{Remark det}
If $(u_0,f)\in \Tilde{O}$ and $\{(\bar{u}^m_0,\bar{f}^m)\}_{m\in \mathbb{N}}$ approximates $(u_0,f)$ in the sense of Theorem \ref{P.L. Lions style}, then
\begin{align*}
    \lVert (u-\bar{u}^m)(t)\rVert^2  & \leq  e^{2t\lVert\nabla \bar{ u}^m\rVert_{L^{\infty}([0,T]\times D)}}(\lVert u_0- \bar{u}_0^m\rVert^2\\ &+2\sqrt{T}\lVert f-\bar{f}^m\rVert_{L^2(0 ,T;H)}\left(\sqrt{2\lVert u_0\rVert^2+4T\lVert f\rVert_{L^2(0,T;H)}^2}+\sqrt{2\lVert \bar{u}_0^m\rVert^2+4T\lVert \bar{f}^m\rVert_{L^2(0,T;H)}^2}\right))\\ & \leq  \frac{1}{m^2} \left(1+2\sqrt{T}\left(\sqrt{2\lVert u_0\rVert^2+4T\lVert f\rVert_{L^2(0,T;H)}^2}+\sqrt{8\lVert u_0\rVert^2+16T\lVert f\rVert_{L^2(0,T;H)}^2+\frac{8+16T}{3}}\right)\right)\\ & \leq  \frac{K\left(T,\lVert u_0\rVert,\lVert f\rVert_{L^2(0,T;H)}\right)}{m^2}.\end{align*}
\end{remark}
Thanks to Lemma \ref{P.L. Lions style} we are able to prove a Kato type inviscid limit result also in the deterministic framework. \\
\begin{theorem}\label{deterministc thm}
If $(u_0,f)\in \Tilde{O}$, $u_0^n\in H,\ f^n\in L^2(0,T;H)$, $$\lim_{n\rightarrow+\infty}\lVert u_0^n-u_0\rVert=0,\ \ \lim_{n\rightarrow +\infty}\lVert f^n-f\rVert_{L^2(0,T;H)}=0. $$  Let $u$ be the solution of the Euler equations with initial condition $u_0$ and external force $f$, $u^n$ be the solution of the deterministic Navier-Stokes equations with viscosity $\nu_n$, initial condition $u^n_0$ and external force $f^n$. If $$\lim_{n\rightarrow+\infty }\nu_n=0,\ \ \lim_{n\rightarrow +\infty}\nu_n\int_0^T \lVert\nabla u^{n}(t) \rVert^2_{L^2(\Gamma_{c\nu_n})}\ dt= 0,  $$ then 
$$\lim_{n \rightarrow +\infty}\sup_{t\in [0,T]}\lVert u^{n}-u\rVert^2= 0. $$
\end{theorem}
\begin{proof}
The proof is an adaptation of previous stochastic arguments, the only novelty is the presence of deterministic external forces. Hence we just give details on the new elements.  \\
Let $\{(\bar{u}_0^m,\bar{f}^m)\}_{m\in\mathbb{N}}$ approximating $(u_0,f)$ in the sense of Theorem \ref{P.L. Lions style} and $\{\bar{u}_m\}_{m\in\mathbb{N}}$ the corresponding solutions of the Euler equations, then for each $t,\ n,\ m$ we have
\begin{align*}
    \lVert u^n(t)-u(t)\rVert^2 & \leq 2 \lVert u^n(t)-\bar{u}^m(t)\rVert^2+2\lVert \bar{u}^m(t)-u(t)\rVert^2\\  & \stackrel{Remark\  \ref{Remark det}}{\leq} \frac{K}{m^2}+2 \lVert u^n(t)-\bar{u}^m(t)\rVert^2.
\end{align*}
For each $m$ and $\delta>0$ fixed, let us introduce the corrector of the boundary layer $v_m$, it satisfies previous estimates. We have at time $t$
\begin{align*}
  \lVert u^n-\bar{u}^m\rVert^2& = \lVert u^{n}\rVert^2+\lVert\bar{u}^m\rVert^2-2\langle u^{n},\bar{u}^m\rangle\\  & \stackrel{energy \ eq.,\ \lVert v_m\rVert_{L^2(D)}\leq K_m\delta^{\frac{1}{2}} }{\leq}\lVert u_0^{n}-u_0\rVert^2+2\lVert u_0\rVert^2+K\lVert u_0^n-u_0\rVert+\lVert\bar{u}_0^m-u_0\rVert^2+K\lVert\bar{u}_0^m-u_0\rVert\\ &+2\int_0^t\langle f^n,u^n\rangle\ ds+2\int_0^t\langle \bar{f}^m, \bar{u}^m\rangle\ ds-2\langle u^{n},\bar{u}^m-v_m\rangle+K_m\delta^{\frac{1}{2}}.
\end{align*}
To analyze the second-last term we use the weak formulation of $u^{n}$, taking $\bar{u}^m-v_m$ as test function. Exploiting the relation $$\langle u_0^n,\bar{u}^m_0\rangle=\lVert u_0\rVert^2+\langle u_0^n-u_0,\bar{u}_0^m-u_0\rangle+\langle u_0^n-u_0,u_0\rangle+\langle u_0,\bar{u}_0^m-u_0\rangle, $$
we get

\begin{align*} -2\langle u^{n}(t),(u^m-v_m)(t)\rangle+2\lVert u_0\rVert^2     & \stackrel{\lVert v_m\rVert_{L^2(D)}\leq K_m\delta^{\frac{1}{2}}}{\leq}   K \lVert u_0^n-u_0\rVert +K\lVert\bar{u}_0^m-u_0\rVert+ K_m\delta^{\frac{1}{2}} \\ & -2\int_0^t\langle u^{n}(s),\partial_s(\bar{u}^m-v_m)(s)\rangle\ ds\\ & +  2\nu_n\int_0^t \langle (-A)^{\frac{1}{2}}u^{n}(s), (-A)^{\frac{1}{2}}(\bar{u}^m-v_m)(s)\rangle\ ds \\ &-\int_0^t 2b(u^{n}(s),(\bar{u}^m-v_m)(s),u^{n}(s))\ ds\\ &-2\int_0^t\langle f^n,\bar{u}^m-v_m\rangle\ ds .
\end{align*}
Moreover
\begin{align*}
    -\langle u^{n}(s),\partial_s(\bar{u}^m-v_m)(s)\rangle & \stackrel{energy \ eq., \ \lVert\partial_t v_m(t)\rVert_{L^2(D)}\leq K_m\delta^{\frac{1}{2}}}{=}-\langle u^{n}(s),\partial_s \bar{u}^m(s)\rangle+K_m\delta^{\frac{1}{2}} \\ & \stackrel{Euler\  eq}{=} \langle u^{n}(s),\nabla\bar{u}^m(s)\nabla\bar{u}^m\rangle-\langle \bar{f}^m(s),u^n(s)\rangle+K_m\delta^{\frac{1}{2}}.
\end{align*}
Thanks to previous relations and noting that $$ b(u^{n},\bar{u}^m,u^{n})-b(u^{n},\bar{u}^{m},\bar{u}^{m})=b(u^{n}-\bar{u}^{m},\bar{u}^{m},u^{n}-\bar{u}^{m})$$ we can continue the estimate of $\lVert u^{n}-\bar{u}^m\rVert^2$:
\begin{align*}
     \lVert u^{n}-\bar{u}^m\rVert^2&  \stackrel{\lVert v_m\rVert_{L^2(D)}\leq K_m\delta^{\frac{1}{2}}}{\leq} 
     K_m\delta^{\frac{1}{2}}+K\lVert u_0^{n}-u_0\rVert^2+\lVert\bar{u}_0^m-u_0\rVert^2+K\lVert\bar{u}_0^m-u_0\rVert +K\lVert u_0^n-u_0\rVert\\  & + 2\int_0^t b(u^n,v_m,u^n)\ ds  +  2\nu_n\int_0^t \langle(-A)^{\frac{1}{2}}u^{n}(s), (-A)^{\frac{1}{2}}(\bar{u}^m-v_m)(s)\rangle\ ds\\ & +2\lVert \nabla\bar{u}^m\rVert_{L^{\infty}([0,T]\times D)}\int_0^t\lVert u^n-\bar{u}^m\rVert^2\ ds \\& +2\int_0^t\langle f^n-\bar{f}^m,u^n-\bar{u}^m\rangle\ ds .
\end{align*}
Arguing as in the stochastic case, we have
$$\int_0^t b(u^n,v_m,u^n)\ ds\stackrel{\lVert\rho^2(t)\nabla v_m(t)\rVert_{L^{\infty}(D)}\leq K_m\delta}{\leq} K_m\delta \int_0^T\lVert \nabla u^n\rVert_{L^2(\Gamma_{\delta})}^2\ ds$$
\begin{align*}
 \nu_n\int_0^t \langle(-A)^{\frac{1}{2}}u^{n}(s), (-A)^{\frac{1}{2}}(\bar{u}^m-v_m)(s)\rangle\ ds  \stackrel{\lVert\nabla v^m(t)\rVert_{L^2(D)}\leq K_m\delta^{-1/2}}{\leq}& \lVert\nabla \bar{u}^m\rVert_{L^{\infty}(0,T;L^2(D))}\nu_n\int_0^T\lVert\nabla u^n\rVert_{L^2(D)}\ ds\\  &  +\nu_nK_m\delta^{-1/2}\int_0^T \lVert\nabla u^n\rVert_{L^2(\Gamma_{\delta})}\ ds.
\end{align*}
For what concern the new term \begin{align*}
    \int_0^t\langle f^n-\bar{f}^m,u^n-\bar{u}^m\rangle\ ds  & \leq  \int_0^t \lVert f^n-\bar{f}^m\rVert \lVert u^n-\bar{u}^m\rVert\ ds\\ & \leq  \sqrt{T}\lVert f^n-\bar{f}^m\rVert_{L^2(0,T;H)}(\lVert u^n\rVert_{L^{\infty}(0,T;H)}+\lVert \bar{u}^m\rVert_{L^{\infty}(0,T;H)})\\  & \stackrel{energy \ eq.}{\leq}  K \left( \lVert f^n-f\rVert_{L^2(0,T;H)} +\lVert \bar{f}^m-f\rVert_{L^2(0,T;H)}\right).
\end{align*}
Thanks to this relations we can continue the estimate of $\lVert u^{n}-\bar{u}^m\rVert^2$:
\begin{align*}
      \lVert u^{n}-\bar{u}^m\rVert^2   & \leq 
     K_m\delta^{\frac{1}{2}}+K\lVert u_0^{n}-u_0\rVert^2+\lVert\bar{u}_0^m-u_0\rVert^2+K\lVert\bar{u}_0^m-u_0\rVert+K\lVert u_0^n-u_0\rVert \\ & + K_m\delta \int_0^T\lVert\nabla u^n\rVert_{L^2(\Gamma_{\delta})}^2\ ds+K \left( \lVert f^n-f\rVert_{L^2(0,T;H)} +\lVert \bar{f}^m-f\rVert_{L^2(0,T;H)}\right) \\ & + 2\lVert\nabla \bar{u}^m\rVert_{L^{\infty}(0,T;L^2(D))}\nu_n\int_0^T\lVert\nabla u^n\rVert_{L^2(D)}\ ds+\nu_nK_m\delta^{-1/2}\int_0^T \lVert\nabla u^n\rVert_{L^2(\Gamma_{\delta})}\ ds  \\ & +2\lVert\nabla\bar{u}^m\rVert_{L^{\infty}([0,T]\times D)}\int_0^t\lVert u^n-\bar{u}^m\rVert^2\ ds .
\end{align*}
Taking $\delta=c\nu_n$, by Gronwall's inequality and Holder's inequality we have
\begin{align*}
      sup_{t\in[0,T]}\lVert u^{n}-\bar{u}^m\rVert^2 &  \leq  (
     K_m\nu_n^{1/2}+K\lVert u_0^{n}-u_0\rVert^2+\lVert\bar{u}_0^m-u_0\rVert^2+K\lVert\bar{u}_0^m-u_0\rVert+K\lVert u_0^n-u_0\rVert\\ & + K_m\nu_n \int_0^T\lVert\nabla u^n\rVert_{L^2(\Gamma_{\delta})}^2\ ds+ K\left( \lVert f^n-f\rVert_{L^2(0,T;H)} +\lVert \bar{f}^m-f\rVert_{L^2(0,T;H)}\right) \\ & + \lVert\nabla \bar{u}^m\rVert_{L^{\infty}(0,T;L^2(D))}K\sqrt{\nu_n}\left(\int_0^T\nu_n\lVert\nabla u^n\rVert_{L^2(D)}^2\ ds\right)^{1/2}\\ &+K_m\left(\nu_n \int_0^T \lVert\nabla u^n\rVert^2_{L^2(\Gamma_{\delta})}\ ds\right)^{\frac{1}{2}} ) e^{2T\lVert\nabla \bar{u}^m\rVert_{L^{\infty}([0,T]\times D)}}.
\end{align*}

Taking the limsup with respect to $n$ of this expression for $m$ fixed we have
\begin{align}\label{estimate final thm det}
    \limsup_{n\rightarrow+\infty} \sup_{t\in[0,T]}\lVert u^{n}-\bar{u}^m\rVert^2\leq & \left(K\lVert\bar{u}^m_0-u_0\rVert+\lVert\bar{u}^m_0-u_0\rVert^2+\lVert \bar{f}^m-f\rVert_{L^2(0,T;H)}\right) e^{2T\lVert\nabla \bar{u}^m\rVert_{L^{\infty}([0,T\times D])}}\\ \leq & \frac{K}{m}+\frac{K}{m^2}.
\end{align}
Coming back to \begin{align*}
    \lVert u^n(t)-u(t)\rVert^2& \leq \frac{K}{m^2}+2 \lVert u^n(t)-\bar{u}^m(t)\rVert^2.
\end{align*}
If we fix $\epsilon>0$ and $\bar{m}$ such that $\frac{5K}{\bar{m}}<\epsilon$, then taking the limsup with respect to $n$ of previous expression for $m=\bar{m}$ we have 
$$\limsup_{n\rightarrow+\infty} \lVert u^n(t)-u(t)\rVert^2\leq \epsilon.$$
We have the thesis from the arbitrariness of $\epsilon$.
\end{proof}
\section*{Acknowledgement}
I want to thank Professor Franco Flandoli for useful discussions and valuable insight into the subject.
  \bibliography{demo}
  \bibliographystyle{abbrv}
\end{document}